\documentclass[11pt]{amsart}
\usepackage{amssymb}
\usepackage{amsmath}
\usepackage{mathrsfs}
\usepackage{stmaryrd}
\usepackage[all]{xy}
\usepackage{color}

\newtheorem{Theorem}{Theorem}

\newtheorem{theorem}{Theorem}[section]
\newtheorem{lemma}[theorem]{Lemma}
\newtheorem{corollary}[theorem]{Corollary}
\newtheorem{proposition}[theorem]{Proposition}

\theoremstyle{definition}
\newtheorem{definition}[theorem]{Definition}

\theoremstyle{remark}
\newtheorem{remark}[theorem]{Remark}

\def\Z{\mathbf{Z}}
\def\C{\mathbf{C}}

\def\Q{\mathbf{Q}}

\def\P{\mathbf{P}}
\def\G{\mathbf{G}}

\let\ms\mathscr

\let\mc\mathcal

\let\ol\overline
\let\ul\underline

\let\lbb\llbracket
\let\rbb\rrbracket

\DeclareMathOperator{\Hom}{Hom}
\DeclareMathOperator{\Tor}{Tor}

\DeclareMathOperator{\Sym}{Sym}

\DeclareMathOperator{\Ind}{Ind}
\DeclareMathOperator{\Spec}{Spec}

\DeclareMathOperator{\avg}{avg}

\DeclareMathOperator{\sgn}{sgn}

\DeclareMathOperator{\Frac}{Frac}
\DeclareMathOperator{\red}{red}
\DeclareMathOperator{\Sub}{Sub}
\DeclareMathOperator{\dSub}{\Delta{}Sub}

\newcommand{\GL}{\mathrm{GL}}

\newcommand{\SL}{\mathrm{SL}}

\newcommand{\Mod}{\mathrm{Mod}}
\renewcommand{\Vec}{\mathrm{Vec}}
\newcommand{\id}{\mathrm{id}}
\newcommand{\bw}[2]{{\bigwedge}^{\! #1}{#2}}
\newcommand{\lw}[2]{{\textstyle \bigwedge}^{\! #1}{#2}}

\newcommand{\Sp}{\mathbf{M}}
\newcommand{\bS}{\mathbf{S}}
\newcommand{\bV}{\mathbf{V}}
\newcommand{\Cfs}{(\mathrm{fs})}
\newcommand{\uotimes}{\, \ul{\otimes} \,}

\def\qbinom#1#2#3{\left[ \!\! \begin{array}{c} #1 \\ #2 \end{array} \!\! \right]_{#3}}

\title{Syzygies of Segre embeddings and $\Delta$-modules}

\author{Andrew Snowden}
\thanks{The author was partially supported by NSF fellowship DMS-0902661.}

\date{June 20, 2011}

\begin{document}

\begin{abstract}
We study syzygies of the Segre embedding of $\P(V_1) \times \cdots \times \P(V_n)$, and prove two finiteness results.
First, for fixed $p$ but varying $n$ and $V_i$, there is a finite list of ``master $p$-syzygies'' from which all other
$p$-syzygies can be derived by simple substitutions.  Second, we define a power series $f_p$ with coefficients in
something like the Schur algebra, which contains essentially all the information of
$p$-syzygies of Segre embeddings (for all $n$ and $V_i$), and show that it is a rational function.  The list of
master $p$-syzygies and the numerator and denominator of $f_p$ can be computed algorithmically (in theory).
The central observation of this paper is that by considering all Segre embeddings at once (i.e., letting $n$ and the
$V_i$ vary) certain structure on the space of $p$-syzygies emerges.  We formalize this structure in the concept of a
$\Delta$-module.  Many of our results on syzygies are specializations of general results on $\Delta$-modules that we
establish.  Our theory also applies to certain other families of varieties, such as tangent and secant varieties of
Segre embeddings.
\end{abstract}

\maketitle
\tableofcontents

\section{Introduction}

Let $V_1, \ldots, V_n$ be finite dimensional complex vector spaces.  The Segre embedding is the map
\begin{displaymath}
\P(V_1) \times \cdots \times \P(V_n) \to \P(V_1 \otimes \cdots \otimes V_n)
\end{displaymath}
taking a tuple of lines to their tensor product.  Despite its fundamental nature, the syzygies of this embedding are
poorly understood (excluding some special cases, such as when $n=2$).  The purpose of this paper is to establish
a general structural theory of these syzygies.  We show that for $p$ fixed, but $n$ and the $V_i$ variable, the
collection of all $p$-syzygies can be given an algebraic structure, and that this structure has favorable
finiteness properties.  One consequence of this result is that almost any question about $p$-syzygies --- such as:
is the module of 13-syzygies of every Segre embedding generated in degrees at most 20? --- can be resolved by a
finite computation (in theory).  In the rest of the introduction, we describe this structure and
present our two main theorems about it.  As the precise statements of these theorems are somewhat technical, we
begin with an informal discussion that attempts to elucidate the essential content of the first theorem.

\subsection{The first theorem (informal)}
\label{ss:informal}

The equations of the Segre embedding have been known for as long as the embeddings themselves have been known, and
are familiar to most anyone who has studied algebraic geometry.  We briefly recall them.  First suppose that $n=2$,
i.e., consider the
embedding of $\P(V_1) \times \P(V_2)$ into $\P(V_1 \otimes V_2)$.  Let $\{x_i\}_{i \in I_1}$ be coordinates on
$\P(V_1)$, let $\{x_j\}_{j \in I_2}$ be coordinates on $\P(V_2)$ and let $\{x_{ij}\}_{(i,j) \in I_1 \times I_2}$ be
coordinates on $\P(V_1 \otimes V_2)$.  Then the equations
\begin{displaymath}
x_{i_1 j_1} x_{i_2 j_2} - x_{i_1 j_2} x_{i_2 j_1} =0
\end{displaymath}
cut out $\P(V_1) \times \P(V_2)$ inside $\P(V_1 \otimes V_2)$.  Now suppose that $n=3$.  Then, in the obvious
notation, we have the equation
\begin{displaymath}
x_{i_1 j_1 k_1} x_{i_2 j_2 k_2} - x_{i_1 j_1 k_2} x_{i_2 j_2 k_1} =0.
\end{displaymath}
This is not symmetrical in the $i$, $j$ and $k$ indices, and so we get two other families of equations by
changing the roles of the indices.  The resulting collection of equations cuts out $\P(V_1) \times \P(V_2)
\times \P(V_3)$ inside of $\P(V_1 \otimes V_2 \otimes V_3)$.  The situation for $n>3$ is similar.

The equations we have just described are quite similar to each other.  Indeed, they can all be recovered from the
single ``master equation''
\begin{displaymath}
x_{12} x_{34}-x_{14} x_{32} =0
\end{displaymath}
by certain substitution rules.  When $n=2$ we obtain all the equations by replacing the numbers 1 and 3 (resp.\ 2
and 4) in the master equation by elements of $I_1$ (resp.\ $I_2$).  When $n=3$ we obtain all the equations
by replacing the numbers 1 and 3 by elements of $I_1 \times I_2$ and 2 and 4 by elements of $I_3$, or by doing
this with the roles of $I_1$, $I_2$ and $I_3$ interchanged.  For general $n$, we obtain the equations for
$\P(V_1) \times \cdots \times \P(V_n)$ by partitioning the set $\{1, \ldots, n\}$ into two disjoint subsets $A$ and $B$
and then replacing 1 and 3 by elements of $\prod_{k \in A} I_k$ and 2 and 4 by elements of $\prod_{k \in B} I_k$.

We now examine the syzygies between equations (2-syzygies).  We again begin with the case $n=2$.  Consider the
determinant
\begin{displaymath}
\left| \begin{array}{ccc}
x_{i_1 j_1} & x_{i_1 j_2} & x_{i_1 j_3} \\
x_{i_1 j_1} & x_{i_1 j_2} & x_{i_1 j_3} \\
x_{i_2 j_1} & x_{i_2 j_2} & x_{i_2 j_3}
\end{array} \right|
\end{displaymath}
The first two rows are repeated, and so the determinant vanishes when considered as a polynomial in the variables
$x_{ij}$.  Now, the $2 \times 2$ minors along the bottom are exactly the previously constructed equations.  Thus
by taking the Laplace expansion of the determinant along the top row we obtain a linear relation between equations,
i.e., a syzygy.  This construction can be modified to produce 2-syzygies for $n>2$, in the same way that the $n=2$
equations were modified to get equations for $n>2$.  These syzygies generate the module of 2-syzygies, for any $n$.

As was the case for equations, all the 2-syzygies just produced can be recovered from a ``master syzygy'' by certain
substitution rules.  The master syzygy is just the one obtained from the Laplace expansion of
\begin{displaymath}
\left| \begin{array}{ccc}
x_{13} & x_{14} & x_{15} \\
x_{13} & x_{14} & x_{15} \\
x_{23} & x_{24} & x_{25}
\end{array} \right|
\end{displaymath}
The substitution rules are similar:  partition $\{1,\ldots, n\}$ into two sets $A$ and $B$ and replace 1 and 2 by
elements of $\prod_{k \in A} I_k$ and 3, 4 and 5 by elements of $\prod_{k \in B} I_k$.

Given these observations, one might ask:  are there master higher syzygies?  An affirmative answer is provided by the
following result, the first of our two main theorems:

\begin{Theorem}[Informal version]
Given $p$, there is a finite list of ``master $p$-syzygies'' from which all $p$-syzygies can be obtained through
the substitution procedures sketched above.
\end{Theorem}
\setcounter{Theorem}{0}

This result is slightly different than what we have seen for equations and syzygies of equations in that there is
no longer a single master syzygy, but rather a finite list of them.  The theorem does not provide any canonical
list of master syzygies, it just asserts the existence of such a list.  The proof of the theorem can easily be
adapted to yield an algorithm which, given $p$, produces a list of master $p$-syzygies.
Unfortunately, this algorithm is fantastically impractical, even for small $p$ --- it involves
linear algebra over a polynomial ring in $p^{2p}$ variables!  Note that the substitution rules do not change the
general form of the syzygy (e.g., all equations are the difference of two monomials involving two variables), and
so the above theorem implies that there are only finitely many ``forms'' of syzygies.

\subsection{The first theorem}

We now give a more formal version of the above discussion.
We begin by explaining precisely what the ``master equation'' is in coordinate-free language.  Once this is
accomplished, generalizing to higher syzygies will be straightforward.  Let $I_n(V_1, \ldots, V_n)$ be the ideal
defining the Segre embedding of $\P(V_1) \times \cdots \times \P(V_n)$.  The $n$ subscript simply indicates that
we are dealing with $n$ vector spaces; it is redundant but provides clarity.  The ideal is generated by its
degree two piece, so we focus on it; we denote it by $I^{(2)}_n(V_1, \ldots, V_n)$.  We now observe that
$I^{(2)}_n$ has the following pieces of structure:
\begin{itemize}
\item[(A1)] Given maps $V_i \to V_i'$, there is a natural map $I^{(2)}_n(V_1, \ldots, V_n) \to I^{(2)}_n(V'_1, \ldots,
V'_n)$.  In other words, $I^{(2)}_n$ is a functor.
\item[(A2)] Given $\sigma \in S_n$, there is a natural map $\sigma_*:I_n^{(2)}(V_1, \ldots, V_n) \to
I_n^{(2)}(V_{\sigma(1)}, \ldots, V_{\sigma(n)})$.  In other words, the functor $I_n^{(2)}$ is $S_n$-equivariant.
\item[(A3)] There is a natural map
\begin{displaymath}
I^{(2)}_n(V_1, \ldots, V_{n-1}, V_n \otimes V_{n+1}) \to I^{(2)}_{n+1}(V_1, \ldots, V_{n-1}, V_n, V_{n+1}).
\end{displaymath}
Geometrically, this map comes from the commutative diagram of Segre embeddings:
\begin{displaymath}
\xymatrix{
\P(V_1) \times \cdots \times \P(V_{n-1}) \times \P(V_n) \times \P(V_{n+1}) \ar[r] \ar[rd] &
\P(V_1) \times \cdots \times \P(V_{n-1}) \times \P(V_n \otimes V_{n+1}) \ar[d] \\
& \P(V_1 \otimes \cdots \otimes V_{n-1} \otimes V_n \otimes V_{n-1}) }
\end{displaymath}
\end{itemize}
Note that (A3) seems to favor the $n$th parameter of $I_n$; however, due to (A2) symmetry is not actually broken.

The Segre embedding of $\P^1 \times \P^1$ into $\P^3$ is defined by a single quadratic equation, and so the space
$I_2^{(2)}(\C^2, \C^2)$ is one dimensional.  The ``master equation'' is simply a non-zero element of this space.
As we now explain, any quadratic equation for any Segre embedding can be obtained from the master equation through the
operations (A1)--(A3).  Let $\{x_1, x_3\}$ be a basis for the first $\C^2$ and $\{x_2, x_4\}$ a basis
for the second $\C^2$, so that $\{x_{12}, x_{14}, x_{32}, x_{34}\}$ is a basis for $\C^2 \otimes \C^2$.  Then the
quadratic polynomial
\begin{equation}
\label{eq:master}
x_{12}x_{34}-x_{14}x_{32}
\end{equation}
is a non-zero element of $I_2^{(2)}(\C^2, \C^2)$, and can be taken as the master equation.  Let $V_1, \ldots, V_n$
be vector spaces and let $\{x^{(k)}_j\}_{j \in I_k}$ be a basis for $V_k$.  Let $A \amalg B$ be a partition of
$\{1, \ldots, n\}$, let $\alpha$ and $\gamma$ be elements of $\prod_{k \in A} I_k$ and let $\beta$ and $\delta$
be elements of $\prod_{k \in B} I_k$.  Then, in the obvious notation,
\begin{equation}
\label{eq:quad}
x_{\alpha \beta} x_{\gamma \delta}-x_{\alpha \delta} x_{\gamma \beta}
\end{equation}
is an element of $I_n^{(2)}(V_1, \ldots, V_n)$, and such elements span $I_n^{(2)}(V_1, \ldots, V_n)$.  Consider the
composite map
\begin{equation}
\label{eq:comp}
I_2^{(2)}(\C^2, \C^2) \to I_2^{(2)} \bigg( \bigotimes_{k \in A} V_k, \bigotimes_{k \in B} V_k \bigg) \to
I_n^{(2)}(V_1, \ldots, V_n).
\end{equation}
The first map in \eqref{eq:comp} is the one coming form (A1) applied to the maps
\begin{displaymath}
\begin{split}
&f:\C^2 \to \bigotimes_{k \in A} V_k, \qquad f(x_1)=x_{\alpha}, \quad f(x_3)=x_{\gamma} \\
&g:\C^2 \to \bigotimes_{k \in B} V_k, \qquad g(x_2)=x_{\beta}, \quad g(x_4)=x_{\delta}.
\end{split}
\end{displaymath}
The second map in \eqref{eq:comp} comes from repeated applications of (A2) and (A3).  The composite map
in \eqref{eq:comp} simply takes an equation in the variables $x_{12}$, $x_{14}$, $x_{32}$, $x_{34}$ and changes
the 1 and 3 to $\alpha$ and $\gamma$ and the 2 and 4 to $\beta$ and $\delta$.  Thus the operations (A1)--(A3)
truly do provide a coordinate-free replacement for the substitution rules of \S \ref{ss:informal}.  In particular,
the master equation \eqref{eq:master} is taken to the equation \eqref{eq:quad}.  This shows that any quadratic equation
can be obtained from the master equation through the operations (A1)--(A3).

As stated, generalizing the above discussion on equations to higher syzygies is straightforward.  However, it will
be convenient to first introduce a definition.  A \emph{$\Delta$-module} is an object $F$ consisting of:
\begin{itemize}
\item[(B1)] For each non-negative integer $n$, an $S_n$-equivariant functor $F_n:\Vec^n \to \Vec$ satisfying a mild
technical condition.  Here, $\Vec$ denotes the category of complex vector spaces.
\item[(B2)] For each $n$, a natural transformation $F_n(V_1, \ldots, V_{n-1}, V_n \otimes V_{n+1}) \to
F_{n+1}(V_1, \ldots, V_{n+1})$.
\end{itemize}
There are certain conditions placed on the maps in (B2) which we do not state here.  It is clear how the above
definition relates to the previous discussion:  $I^{(2)}$ is a $\Delta$-module.  (Of course, (B1) takes the place of
both (A1) and (A2), while (B2) corresponds to (A3).)  However, this definition is
inelegant, and thus technically inconvenient.  An equivalent (but much nicer) definition is the following.  Let
$\Vec^{\Delta}$ be the following category:
\begin{itemize}
\item An object is a pair $(V, L)$ consisting of a finite set $L$ and, for each $x \in L$, a vector space
$V_x$.  We think of $(V, L)$ as a family of vector spaces.
\item A morphism $(V, L) \to (V', L')$ consists of a surjection $L' \to L$ together with a linear map
$V_x \to \bigotimes_{y \mapsto x} V'_y$, for each $x \in L$.
\end{itemize}
A $\Delta$-module is then simply a functor $F:\Vec^{\Delta} \to \Vec$ (satisfying a mild technical condition).  The
functor $F_n$ in the previous formulation is obtained by restricting $F$ to families with $n$ elements.

We now precisely define our spaces of syzygies and explain how they admit the structure of a $\Delta$-module.
Let $(V, L)$ be a family of vector spaces.  Let $R(V, L)$ be the projective coordinate ring of $\prod_{x \in L}
\P(V_x)$ and let $P(V, L)$ be the projective coordinate ring of $\P(\bigotimes_{x \in L} V_x)$.  Precisely,
\begin{displaymath}
\begin{split}
P(V, L)&=\bigoplus_{n=0}^{\infty} P^{(n)}(V, L), \qquad P^{(n)}(V, L)=\Sym^n \bigg( \bigotimes_{x \in L} V_x \bigg) \\
R(V, L)&=\bigoplus_{n=0}^{\infty} R^{(n)}(V, L), \qquad R^{(n)}(V, L)=\bigotimes_{x \in L} \Sym^n(V_x).
\end{split}
\end{displaymath}
The Segre embedding corresponds to a surjection of rings $P(V, L) \to R(V, L)$.  Put
\begin{displaymath}
F_p(V, L)=\Tor^{P(V, L)}_p(R(V, L), \C),
\end{displaymath}
where $\C$ is given the structure of a $P(V, L)$-module by letting all positive degree elements act by 0.  We call
$F_p(V, L)$ the space of $p$-syzygies.  This terminology is justified for the following reason:  if
\begin{displaymath}
\cdots M_1 \to M_0 \to R(V, L) \to 0
\end{displaymath}
is a minimal free resolution of $R(V, L)$ by finite free $P(V, L)$-modules, then $M_p$ is isomorphic to
$P(V, L) \otimes_{\C} F_p(V, L)$.  (This is a general fact about minimal free resolutions and has nothing to do
with the specifics of the Segre embedding.)  The functoriality of $F_p$ with respect to maps in
$\Vec^{\Delta}$ stems from a corresponding functoriality of $R$ and $P$ (as mentioned in the discussion of the
properties (A1)--(A3)), and the basic properties of $\Tor$.  Thus each $F_p$ is a $\Delta$-module.  Note that
$F_1$ is identified with $I^{(2)}$.

Before precisely stating our first theorem, we must make one more definition.  Let $F$ be a $\Delta$-module.
An \emph{element} of $F$ is an element of $F(V, L)$ for some $(V, L)$.  Given a set $S$ of elements of $F$, there
is a minimal $\Delta$-submodule $F'$ of $F$ containing $S$.  We call $F'$ the submodule of $F$ \emph{generated} by
$S$.  We say that $F$ is \emph{finitely generated} if it is generated by some finite set of elements.  We now
come to our first main theorem:

\begin{Theorem}
\label{thm2}
The $\Delta$-module $F_p$ is finitely generated.
\end{Theorem}

Of course, the ``master syzygies'' are just generators for $F_p$.  As previously mentioned, our proof of this
theorem provides an algorithm to compute a set of generators.

\subsection{The second theorem}

Theorem~\ref{thm2} tells us that all $p$-syzygies can be derived from a finite list of $p$-syzygies via some simple
operations.  However, this by itself tells us very little about the structure of the space $F_p(V, L)$ for particular
values of $(V, L)$.  Our second theorem attempts to close this gap.  The full version of the theorem requires some
set-up, so we begin by stating a simple, elementary version of it:

\begin{Theorem}[Preliminary version]
Fix non-negative integers $p$ and $d$.  Then the series
\begin{displaymath}
\sum_{n=1}^{\infty} \dim F_{p,n}(\C^d, \ldots, \C^d) \cdot t^n
\end{displaymath}
is a rational function.
\end{Theorem}
\setcounter{Theorem}{1}

The series in the above theorem counts the number of $p$-syzygies of the Segre embedding of $(\P^{d-1})^n$.  Thus
the result states that these numbers vary in a predictable manner with $n$.

The above preliminary version of the theorem falls short of what we would like, in three ways.
First, it only deals with syzygies of $(\P^{d-1})^n$, rather than more general products.  Second, it only gives us the
dimension of $F_{p,n}$, rather than the dimensions of its graded pieces.  And third, and most importantly, it gives no
information on the action of the general linear groups on the spaces of syzygies.  The full version of the theorem
addresses all of these issues.

Before properly stating the theorem, we must introduce a generating function.  For a partition $\lambda$,
write $\bS_{\lambda}$ for the corresponding Schur functor.  (See \S \ref{ss:schur} for a
review of this theory.)  By general theory, we have a decomposition
\begin{displaymath}
F_{p,n}(V_1, \ldots, V_n)
=\bigoplus_{i \in I_{p, n}} \bS_{\lambda_{1,i}}(V_1) \otimes \cdots \otimes \bS_{\lambda_{n,i}}(V_n)
\end{displaymath}
for some index set $I_{p, n}$ and partitions $\lambda_{i,j}$ (depending on $n$ and $p$).  It is not difficult to show
that $I_{p, n}$ is finite.  The above decomposition holds as functors $\Vec^n \to \Vec$, i.e, the set $I_{p, n}$ and
the partitions $\lambda_{i,j}$ do not depend on $(V_1, \ldots, V_n)$.  Define
\begin{displaymath}
f_{p,n}^*=\sum_{i \in I_{p, n}} s_{\lambda_{1,i}} \cdots s_{\lambda_{n,i}},
\end{displaymath}
regarded as a polynomial in (commuting) formal variables $s_{\lambda}$.  The functor $F_{p,n}$ can be
recovered from $f_{p,n}^*$ --- allowing the variables to commute does not lose information since the functor
$F_{p,n}$ has an $S_n$-equivariance.  (Note, however, that the data of the $S_n$-equivariance is not recorded in
$f_{p,n}^*$.)  Now define
\begin{displaymath}
f_p^*=\sum_{n=1}^{\infty} f_{p,n}^*,
\end{displaymath}
a power series in the variables $s_{\lambda}$.  A priori, there could be infinitely many variables occurring in
$f_p^*$; we will show that this is not the case.  Our second theorem is then:

\begin{Theorem}
\label{thm3}
The series $f_p^*$ is a rational function of the $s_{\lambda}$.
\end{Theorem}

This theorem completely encompasses the preliminary version, but is much stronger.  Let us give an example to
illustrate how to extract information from $f_p^*$.  From the computations of \S \ref{ss:smallp}, we obtain
\begin{displaymath}
f_1^*=\frac{1-s}{(1-s)^2-w^2} - \frac{1}{1-s},
\end{displaymath}
where $s=s_{(2)}$ and $w=s_{(1,1)}$.  Developing the above into a power series, we obtain
\begin{displaymath}
f_1^*= w^2  + 3sw^2 + (6w^2s^2+w^4) + \cdots
\end{displaymath}
The second order term of this series tells us that $F_{1,2}(V_1,V_2)$ is isomorphic
to $\lw{2}{V_1} \otimes \lw{2}{V_2}$.  The third order terms tells us that $F_{1,3}(V_1,V_2,V_3)$ is isomorphic to
\begin{displaymath}
\big( \Sym^2{V_1} \otimes \lw{2}{V_2} \otimes \lw{2}{V_3} \big) \oplus
\big( \lw{2}{V_1} \otimes \Sym^2{V_2} \otimes \lw{2}{V_3} \big) \oplus
\big( \lw{2}{V_1} \otimes \lw{2}{V_2} \otimes \Sym^2{V_3} \big).
\end{displaymath}
In general, we can recover the isomorphism class of $F_{p,n}(V_1, \ldots, V_n)$ as a representation of $\GL(V_1)
\times \cdots \times \GL(V_n)$ from the order $n$ term of $f_{p,n}^*$.  In particular, we can recover the dimensions of
the graded pieces of $F_{p,n}(V_1, \ldots, V_n)$ from $f_{p,n}^*$, as these pieces can be detected by the group action.

As with Theorem~\ref{thm2}, the proof of Theorem~\ref{thm3} provides an algorithm to compute the numerator and
denominator of $f_p^*$.  Thus, essentially all information about all $p$-syzygies of Segre embeddings can be
determined from a single straightforward calculation.  Unfortunately, this algorithm has complexity comparable to
the previous one and is worthless in practice.

The leading term of $f_p^*$ is known by Lascoux's work (\S \ref{ss:lascoux}).  We will compute a
certain Euler characteristic involving the $f_p^*$'s (\S \ref{ss:euler}).  By known vanishing results, this allows us
to compute $f_1^*$, $f_2^*$, $f_3^*$ and part of $f_4^*$ (\S \ref{ss:smallp}).  In particular, this gives an
essentially complete description of the module of $p$-syzygies for $p=1,2,3$.  We have not been able to compute beyond
these examples, however.

\subsection{Syzygies of $\Delta$-schemes}

Given a finite family of vector spaces $(V_1, \dots, V_n)$ we have the Segre variety $X_n(V_1, \ldots, V_n) \subset
\P(V_1 \otimes \cdots \otimes V_n)$.  The variety $X_n$ is functorial in the $V_i$'s and $S_n$-equivariant.
Furthermore, we have inclusions
\begin{equation}
\label{eq:incl}
X_{n+1}(V_1, \ldots, V_{n+1}) \subset X_n(V_1, \ldots, V_{n-1}, V_n \otimes V_{n+1}).
\end{equation}
These are in fact the only properties of the Segre embedding we need for our theory to apply, as we now explain.

A \emph{$\Delta$-scheme} $X$ is a rule assigning to each finite family of vector spaces $(V_1, \ldots,
V_n)$ a closed subscheme $X_n(V_1, \ldots, V_n)$ of $\P(V_1 \otimes \cdots \otimes V_n)$ which is functorial,
$S_n$-equivariant and such that the inclusion \eqref{eq:incl} holds.  The Segre embeddings constitute an example of a
$\Delta$-scheme, but there are many others:  for instance, the secant varieties of the Segre are $\Delta$-schemes.

Let $X$ be a $\Delta$-scheme.  Let $F_p(V, L)$ be the space of $p$-syzygies of the embedding $X(V, L) \to
\P(\bigotimes_{x \in L} V_x)$.  Then $F_p$ forms a $\Delta$-module.  Let $F_p^{(d)}$ denote the degree $d$ piece of
$F_p$ and let $(f_p^*)^{(d)}$ denote the series defined from $F_p^{(d)}$.  Then we have the following result:

\begin{Theorem}
\label{thm4}
The $\Delta$-module $F_p^{(d)}$ is finitely generated, and the series $(f_p^*)^{(d)}$ is rational.
\end{Theorem}

This result is slightly weaker than the two main theorems for Segre embeddings, since it only applies to the
graded pieces of $F_p$.  To get comparable results, one would need to know that $F_p$ is supported in only finitely
many degrees.  This is true in the case of Segres, but probably not true for a general $\Delta$-scheme.  In
\S \ref{s:finlev} we define a class of $\Delta$-schemes, those of \emph{finite level}, for which it seems plausible
that $F_p$ is supported in finitely many degrees.

\subsection{Outline of proofs}
\label{ss:outline}

We begin with a general study of $\Delta$-modules.  The heart of the paper is occupied with the proof of two
theorems:  (a) a small $\Delta$-module is noetherian; and (b) the Hilbert series of a small $\Delta$-module is
rational.  ``Small'' $\Delta$-modules form a certain subclass of $\Delta$-modules.  The class of small $\Delta$-modules
is closed under subquotients.   Theorem~\ref{thm4}
is then a formal consequence of (a) and (b) and the following observation:  there is a complex of small
$\Delta$-modules (a Koszul complex) whose homology is $F_p$.  The two stronger theorems for Segre embeddings are
obtained by applying the well-known result that the $p$-syzygies of the Segre embedding are supported in only
finitely many degrees.

We have not been able to find a direct method of studying $\Delta$-modules.  Instead, we access them through the
following ``ladder'':
\begin{displaymath}
\{\textrm{graded modules}\} \leftrightarrow \{\textrm{modules in $\Sym(\Vec)$}\} \leftrightarrow
\{\textrm{modules in $\Sym(\mc{S})$}\} \leftrightarrow \{\textrm{$\Delta$-modules}\}
\end{displaymath}
We define the middle two categories below.  The arrows here do not mean anything precise, only that the two
categories are related.  We prove our results about $\Delta$-modules by starting with the corresponding results
for graded modules (which are easy and well-known) and moving up the ladder.  This process is actually fairly
explicit:  for example, the Hilbert series of a $\Delta$-module is identified with the ($G$-equivariant) Hilbert
series of a module over a graded ring, which is constructed in an explicit manner from the $\Delta$-module.
This identification is the reason that the Hilbert series of a $\Delta$-module --- at least one which is described in
a reasonable manner --- is algorithmically computable.

\subsection{Future directions}

We prove a number of purely algebraic results (about $\Delta$-module, twisted commutative algebras, etc.) which seem to
scratch the surface of a much larger theory.  We hope to pursue some of this theory in the future.  We also hope
to refine our results on syzygies --- for instance, what can one say about the rational functions $f_p^*$?  A list
of some questions and problems along these lines appears in \S \ref{s:ques}.

Some of the results and methods we develop seem to apply to slightly different settings.  For example, if $X$ is a
projective variety with an action of a reductive group $G$, then the syzygies of the GIT quotients $X^n/\!/G$ (as
$n$ varies) have some of the algebraic structure that we observe here for the Segre varieties.  One can hope that
there are finiteness results in this setting.  We will return to this topic in a future paper.

\subsection*{Acknowledgments}

I would like to thank Aldo Conca, Tony Geramita, Ben Howard, Sarah Kitchen, Rob Lazarsfeld, Diane Maclagan, John
Millson, Lawrence O'Neil, Steven Sam and Ravi Vakil for helpful discussions.  In particular, I would like to thank
Sarah Kitchen, Rob Lazarsfeld and Ravi Vakil for their comments on drafts of this paper.

\section{$\Delta$-modules}

The purpose of this section is to introduce the algebraic objects that we will use in the rest of the paper, most
notably $\Delta$-modules.  We begin by quickly reviewing the theory of Schur functors and define the
Schur algebra $\mc{S}$, which we regard as a category.  We then discuss symmetric powers of semi-simple
abelian categories.  This operation is not strictly necessary for our purposes, but clarifies some later
definitions.  We then introduce the categories $\Sym(\Vec)$ and $\Sym(\mc{S})$ and give explicit models for
them which do not use the symmetric power construction.  Finally we discuss $\Delta$-modules, which are
objects of $\Sym(\mc{S})$ with an additional piece of structure.  The most important results of this section are
various statements that certain types of objects are noetherian.

Algebras in $\Sym(\Vec)$ are known as ``twisted commutative algebras,'' and there is some literature on them
(see \cite{Barratt}, \cite[Ch.~4]{Joyal}, \cite{GinzburgSchedler}); closely related is Joyal's theory of
tensorial species.  However, the existing literature --- at least that which we are aware of --- studies these
objects from a perspective different from our own.  In particular, we have not encountered the noetherian result we
prove about them in the literature.  The categories $\Sym(\mc{S})$ and $\Mod_{\Delta}$ are a bit more esoteric, and we
do not know of any occurrence of them in the literature.

\subsection{The Schur algebra}
\label{ss:schur}

We quickly review what we need about the Schur algebra.  Further discussion can be found in \cite[\S 6.1]{FH}.  Let
$\Vec$ denote the category of complex vector spaces.  A functor $\Vec \to \Vec$ is ``nice'' if it appears as a
constituent of a functor of the form $V \mapsto V^{\otimes n}$, or a (possibly infinite) direct sum of such functors.
Essentially every functor $\Vec \to \Vec$ one encounters naturally that does not use duality is nice; an example of a
non-nice functor is the double dual.  Let $\mc{S}$ denote the category of all nice functors.  It is a semi-simple
abelian tensor category which we call the \emph{Schur algebra}.  (By ``semi-simple'' here we mean that every object
is a possibly infinite direct sum of simple objects.)  The tensor product is the point-wise one:  $(F \otimes G)(V)
=F(V) \otimes G(V)$.

Let $\lambda$ be a partition of $n$; we denote this by $\vert \lambda \vert=n$ or $\lambda \vdash n$.  Let
$\Sp_{\lambda}$ be the irreducible complex representation of the symmetric group $S_n$ associated to $\lambda$.
For a vector space $V$ put $\bS_{\lambda}(V)=(V^{\otimes n} \otimes
\Sp_{\lambda})_{S_n}$, where the subscript denotes coinvariants and $S_n$ acts on $V^{\otimes n}$ by permuting tensor
factors.  Then $\bS_{\lambda}$ belongs to $\mc{S}$ and is a simple object; furthermore, every simple object is
isomorphic to $\bS_{\lambda}$ for a unique $\lambda$.  We use the convention that the partition $\lambda=(n)$ gives the
functor $\bS_{\lambda}=\Sym^n$ while the partition $\lambda=(1, \ldots, 1)$ gives the functor $\bS_{\lambda}=\lw{n}{}$.
We identify partitions with Young diagrams by the convention that $\lambda=(n)$ has one row and $\lambda=(1, \ldots,
1)$ has $n$ rows.

We will also need multivariate functors of vector spaces.  A functor $\Vec^k \to \Vec$ is ``nice'' if it appears as
a constituent of a functor of the form $(V_1, \ldots, V_k) \mapsto V_1^{\otimes n_1} \otimes \cdots \otimes
V_k^{\otimes n_k}$, or a (possibly infinite) direct sum of such functors.  The category of all nice functors is
again a semi-simple abelian tensor category, and is naturally equivalent to $\mc{S}^{\otimes k}$.  This means that any
nice functor $F:\Vec^k \to \Vec$ can be written as
\begin{displaymath}
F(V_1, \ldots, V_k)=\bigoplus_{i \in I} \bS_{\lambda_{1, i}}(V_1) \otimes \cdots \otimes \bS_{\lambda_{k, i}}(V_k)
\end{displaymath}
for some index set $I$ and partitions $\lambda_{i, j}$.  This expression is unique.  We say that a partition
$\lambda$ \emph{occurs} in $F$ if it is amongst the $\lambda_{i, j}$.  If $F:\Vec \to \Vec$ is a nice functor then
the functors $\Vec^2 \to \Vec$ given by mapping $(V, W)$ to $F(V \oplus W)$ and $F(V \otimes W)$ are both nice.
We thus get a co-addition map $a^*$ and a co-multiplication map $m^*$ from $\mc{S}$ to $\mc{S}^{\otimes 2}$.

Let $\mc{A}$ be a $\C$-linear abelian tensor category.  (Our tensor categories are always symmetric, i.e., there is a
given involution of functors $A \otimes B \to B \otimes A$.)  There is then a natural action of $\mc{S}$ on $\mc{A}$.
Given a partition $\lambda$ of $n$ and an object $A$ of $\mc{A}$, the object $\bS_{\lambda}(A)$ is given by
$(A^{\otimes n} \otimes \Sp_{\lambda})_{S_n}$.

\subsection{Symmetric powers of abelian categories}
\label{ss:symcat}

We now describe how one can form the symmetric algebra on semi-simple abelian categories.  The reader who is not
comfortable with this discussion need not worry:  all categories that we eventually use will admit concrete
descriptions.  We provide this discussion only to give some context and motivation for later constructions.

Let $\mc{A}$ be a semi-simple $\C$-linear abelian category.  One can then make sense of the tensor power
$\mc{A}^{\otimes n}$ (see \cite[\S 5]{Deligne} for a discussion of tensor products of abelian categories).
We define $\Sym^n(\mc{A})$ to be the category $(\mc{A}^{\otimes n})^{S_n}$, where the
superscript denotes homotopy invariants (equivariant objects).  Thus an object of $\Sym^n(\mc{A})$ is an object
$A$ of $\mc{A}^{\otimes n}$ together with an isomorphism $\sigma^* A \to A$ for each $\sigma \in S_n$,
satisfying the obvious compatibility conditions.  We define $\Sym(\mc{A})$ to be the sum of the categories
$\Sym^n(\mc{A})$ over $n \ge 0$.

The category $\Sym(\mc{A})$ has a natural tensor structure.  Multiplication involves averaging (induction).
Precisely, let $\uotimes$ denote the usual
concatenation tensor product $\mc{A}^{\otimes n} \otimes \mc{A}^{\otimes m} \to \mc{A}^{\otimes (n+m)}$.  If $A$ is an
object of $(\mc{A}^{\otimes n})^{S_n}$ and $B$ of $(\mc{A}^{\otimes m})^{S_m}$ then $A \uotimes B$ naturally has an
$S_n \times S_m$ equivariance.  One can then form the induction $\Ind_{S_n \times S_m}^{S_{n+m}}(A \uotimes B)$, which
is an object of $(\mc{A}^{\otimes (n+m)})^{S_{n+m}}$.  This is the product of $A$ and $B$ in $\Sym(\mc{A})$, which we
denote by $A \otimes B$.  Note that if $\mc{A}$ itself has a tensor structure then $(\mc{A}^{\otimes n})^{S_n}$ does as
well, which gives an alternate tensor structure on $\Sym(\mc{A})$.  When present, we will denote this tensor product by
$\boxtimes$ and call it the \emph{point-wise} tensor product.

\begin{remark}
The above definition looks more like the divided power algebra than symmetric algebra.  However, one can verify it
has the correct universal property.  We believe that in the setting of $\C$-linear abelian tensor
categories, the right analogue of divided powers is an action of the Schur algebra, with $\Sym^n$ taking the place
of $\gamma_n$.  Since all such categories have a natural action of the Schur algebra, they can all be considered
to have divided powers.  Thus the symmetric and divided power algebras on $\mc{A}$ coincide.  This is a conceptual
reason explaining why we can use invariants (rather than coinvariants) to form the symmetric algebra.
\end{remark}

As $\Sym(\mc{A})$ is an abelian tensor category, we can speak of algebras in it and modules over algebras.  (Algebras
will always be commutative, associative and unital.)  Let $A$ be an algebra.  We say that $A$ is \emph{finitely
generated} (as an algebra) if it is a quotient of $\Sym(F)$ for some finite length object $F$ of $\Sym(\mc{A})$.
Similarly, we say that an $A$-module $M$ is \emph{finitely generated} if it is a quotient of $A \otimes F$ for some
finite length $F$.  We say that $M$ is \emph{noetherian} if every ascending chain of submodules stabilizes.  We say
that $A$ is \emph{noetherian} (as an algebra) if every finitely generated $A$-module is noetherian.

The category $\Sym(\mc{A})$ is again semi-simple abelian, and its simple objects can be described easily.  Let $A$
be a simple object of $\mc{A}$ and let $\Sp_{\lambda}$ be an irreducible representation of $S_n$.  Then $\Sp_{\lambda}
\otimes A^{\uotimes n}$ has a natural $S_n$-equivariance and so defines an object of $\Sym^n(\mc{A})$; in fact, this is
nothing other than $\bS_{\lambda}(A)$ in the category $\Sym(\mc{A})$.  This object is simple, and all simple objects
are of the form $\bigotimes \bS_{\lambda_i}(A_i)$ where the $A_i$ are mutually non-isomorphic simple objects of
$\mc{A}$.

Let $K(-)$ denote the Grothendieck group of an abelian category, tensored with $\Q$.  We have a map
\begin{displaymath}
K(\Sym^n(\mc{A})) \to \Sym^n(K(\mc{A}))
\end{displaymath}
defined as follows:  first apply the functor $\Sym^n(\mc{A}) \to \mc{A}^{\otimes n}$, then use the identification
$K(\mc{A}^{\otimes n})=K(\mc{A})^{\otimes n}$, then project $K(\mc{A})^{\otimes n} \to \Sym^n(K(\mc{A}))$ and finally
divide by $n!$.  For $A \in \Sym^n(\mc{A})$ we let $[A]$ denote the corresponding class in $\Sym^n(K(\mc{A}))$.  We
also put $[A]^*=n! [A]$.  The above map is not injective, but only forgets the $S_n$-equivariant structure:
if $[A]=[B]$ then $A$ and $B$ have isomorphic images in $\mc{A}^{\otimes n}$.  Summing the above maps over $n$, we get
a map
\begin{displaymath}
K(\Sym(\mc{A})) \to \Sym(K(\mc{A})).
\end{displaymath}
Again, for $A \in \Sym(\mc{A})$ we let $[A]$ denote the corresponding class in $\Sym(K(\mc{A}))$; we define $[A]^*$ in
the obvious manner.  We have $[A \otimes B]=[A][B]$; the modified class $[-]^*$ does not satisfy this.  As an example,
if $A$ is an object of $\mc{A}$ then $[\Sym^n(A)]=\tfrac{1}{n!} [A]^n$ (which leads to the beautiful formula
$[\Sym(A)]=\exp([A])$).

We will often need to use the completion of $\Sym$, both in the setting of $\Z$-modules and abelian categories.  Whereas
$\Sym(\Z^n)$ is the ring of polynomials in $n$ variables, the completion is the ring of power series in $n$ variables.
We will not bother to introduce extra notation for the completion, as there should be no confusion.

\subsection{The category $\Sym(\Vec)$}

The following abelian tensor categories are equivalent:
\begin{enumerate}
\item The category $\Sym(\Vec)$.
\item The category of functors $\Cfs \to \Vec$, where $\Cfs$ is the category whose objects are finite sets and
whose morphisms are bijections.  Multiplication is given by convolution, using the monoidal structure on $\Cfs$ given
by disjoint union.  Precisely, if $V$ and $W$ are two functors $\Cfs \to \Vec$ then
\begin{displaymath}
(V \otimes W)_L=\bigoplus_{L=A \amalg B} V_A \otimes W_B,
\end{displaymath}
where the sum is over all partitions of $L$ into two disjoint subsets $A$ and $B$.  (We use subscripts to indicate the
value of the functor on a set).
\item The category of sequences $(V_n)_{n \ge 0}$, where $V_n$ is a vector space with an action of $S_n$.
Multiplication is given by the formula
\begin{displaymath}
(V \otimes W)_n=\bigoplus_{n=i+j} \Ind_{S_i \times S_j}^{S_n} (V_i \otimes W_j).
\end{displaymath}
\item The Schur algebra $\mc{S}$.  Multiplication is the point-wise tensor product.
\item The full subcategory of the representation category of $\GL(\infty)$ on objects which appear as a constituent
of a direct sum of tensor powers of the standard representation $\C^{\infty}$.  Here $\GL(\infty)$ is the union
of $\GL(n, \C)$ over $n \ge 1$.  Multiplication is the usual tensor product of representations.
\end{enumerate}
We briefly describe the various equivalences.  The equivalence between (b) and (c) is clear.  Since $\Vec^{\otimes
n}=\Vec$, the category of $S_n$-equivariant objects in $\Vec^{\otimes n}$ is just the representation category of $S_n$;
this gives the equivalence between (a) and (c).  The equivalence between (c) and (d) is through Schur-Weyl duality.
Precisely, given a sequence $(V_n)$ in the category (c), let $S:\Vec \to \Vec$ be the functor taking a vector space
$W$ to
\begin{displaymath}
S(W)=\bigoplus_{n \ge 0} (W^{\otimes n} \otimes V_n)_{S_n}.
\end{displaymath}
Then $(V_n) \mapsto S$ is the equivalence.  Regarding $\Sp_{\lambda}$ as an object of (c)
supported at $n=\vert \lambda \vert$, this equivalence maps $\Sp_{\lambda}$ to $\bS_{\lambda}$.
Finally, the equivalence of (d) and (e) is given by evaluating a Schur functor on $\C^{\infty}$.
We regard (b)--(e) as ``models'' for the category $\Sym(\Vec)$.  We name them the ``standard,'' ``sequence,''
``Schur,'' and ``$\GL$'' models, respectively.  (There are even more models for this category, but they are not
relevant for our purposes.)  We now come to an important definition:

\begin{definition}
A \emph{twisted commutative algebra} is an algebra in the category $\Sym(\Vec)$.
\end{definition}

As always, ``algebra'' means commutative, associative and unital.  In the standard model, a twisted commutative algebra
is a functor $A:\Cfs \to \Vec$ together with a multiplication map $A_L \otimes A_{L'} \to A_{L \amalg L'}$ satisfying
the appropriate identity.  In the $\GL$-model, a twisted commutative algebra is just a commutative $\C$-algebra, in the
usual sense, with an action of $\GL(\infty)$ by algebra homomorphisms (under which the algebra forms an appropriate
kind of representation).  The notion of a module over a twisted commutative algebra is evident.

Let $V$ be an object of $\Sym(\Vec)$, in the standard model.  By an \emph{element} of $V$ we mean an element of $V_L$
for some $L$; we then call $\# L$ its \emph{order}.  (It would be more natural to use the term ``degree'' here, but
we want to reserve that term for future use.)  Given a collection $S$ of elements of $V$ there is a unique
minimal subobject of $V$ containing $S$; we call it the subspace of $V$ \emph{generated} by $S$.
Similarly, if $A$ is a twisted commutative algebra and $S$ is a collection of elements of $A$ then one can speak of
the subalgebra of $A$ generated by $S$.  One can do the same for modules over $A$.  We therefore have a notion of
``finitely generated'' for algebras and modules; this agrees with the one defined in \S \ref{ss:symcat} using finite
length objects.

Let $U$ be a vector space.  We let $U\langle 1 \rangle$ be the object of $\Sym(\Vec)$ which is $U$ in order 1 and
0 in other orders.  We put $A=\Sym(U\langle 1 \rangle)$; this is the most important twisted commutative algebra
in this paper.  Here are precise descriptions of $U\langle 1 \rangle$ and $A$ in the various models:
\begin{itemize}
\item In the standard model, $U\langle 1 \rangle$ assigns to a finite set $L$ the vector space $U$ if $\# L=1$ and
the vector space 0 otherwise.  The algebra $A$ assigns to a finite set $L$ the vector space $U^{\otimes L}$.  The
multiplication map $A_L \otimes A_{L'} \to A_{L \amalg L'}$ is concatenation of tensors.  (Note:  $U^{\otimes L}$
is isomorphic as a vector space to $U^{\otimes n}$, where $n=\# L$, but is functorial in $L$.
It can be defined as the universal vector space equipped with a multi-linear map from $U \times L$.  We think of
the factors of a pure tensor in $U^{\otimes L}$ as being labeled with elements of $L$.)
\item In the sequence model, $U\langle 1 \rangle$ is the vector space $U$ in order 1 and 0 in other orders.  The
algebra $A$ is given by $A_n=U^{\otimes n}$, with its usual $S_n$ action; the multiplication map $A_i \otimes A_j
\to A_{i+j}$ is again concatenation of tensors.
\item In the Schur model, $U\langle 1 \rangle$ is the functor $U \otimes \Sym^1$.  The algebra $A$ is given by
$\Sym(U \otimes \Sym^1)$.  If $V$ is a vector space then $(U\langle 1 \rangle)(V)=U \otimes V$ and $A(V)=\Sym(U
\otimes V)$.
\item In the $\GL$-model, $U\langle 1 \rangle$ is $U \otimes \C^{\infty}$.  The algebra $A$ is $\Sym(U \otimes
\C^{\infty})$.  Thus, in this model, $A$ is a polynomial ring in infinitely many variables with $\GL(\infty)$ acting
by linear substitutions.
\end{itemize}
The following result underlies everything else in the paper.  We learned the key ideas of its proof from Harm Derksen
and Ben Howard.

\begin{theorem}
\label{tc:noeth}
A twisted commutative algebra finitely generated in order 1 is noetherian.
\end{theorem}

\begin{proof}
A twisted commutative algebra which is finitely generated in order 1 is a quotient of $\Sym(U \langle 1 \rangle)$
for some finite dimensional vector space $U$.  Thus it suffices to show that such algebras are noetherian.  Fix $U$
and put $A=\Sym(U \langle 1 \rangle)$.
In the $\GL$-model, $A$ is given by $\Sym(U \otimes \C^{\infty})$.  Let $d=\dim{U}$.  The following lemma
shows that contraction from $\GL(\infty)$-stable ideals of $\Sym(U \otimes \C^{\infty})$ to ideals of
$\Sym(U \otimes \C^d)$ is injective.  As $\Sym(U \otimes \C^d)$ is a polynomial algebra in finitely many variables,
it is noetherian.  We conclude that $A$ is a noetherian module over itself.  A slight modification of this argument
shows that any finitely generated $A$-module is noetherian, which proves that $A$ is noetherian as an algebra.
\end{proof}

\begin{lemma}[Weyl]
Let $U$ be a vector space of dimension $d$.  If $W \subset \Sym^k(U \otimes \C^{\infty})$ is a $\GL(\infty)$-stable
subspace then $W$ is generated as a $\GL(\infty)$-module by $W \cap \Sym^k(U \otimes \C^d)$.
\end{lemma}

\begin{proof}
Using the formula for the symmetric power of a tensor product (see \cite[Exercise~6.11(b)]{FH}), we obtain a diagram
\begin{displaymath}
\xymatrix{
\Sym^k(U \otimes \C^d) \ar@{=}[r] \ar@{^(->}[d] &
\bigoplus \bS_{\lambda}(U) \otimes \bS_{\lambda}(\C^d) \ar@{^(->}[d] \\
\Sym^k(U \otimes \C^{\infty}) \ar@{=}[r] &
\bigoplus \bS_{\lambda}(U) \otimes \bS_{\lambda}(\C^{\infty}) }
\end{displaymath}
The sums are taken over the partitions $\lambda$ of $k$.  Now, let $W$ be a $\GL(\infty)$-stable subspace
of $\Sym^k(U \otimes \C^{\infty})$.  Since the $\bS_{\lambda}(\C^{\infty})$ are mutually non-isomorphic irreducible
$\GL(\infty)$-modules, we
can write $W=\bigoplus W_{\lambda} \otimes \bS_{\lambda}(\C^{\infty})$, where $W_{\lambda}$ is a subspace of
$\bS_{\lambda}(U)$.  Note that if $\lambda$ has more than $d$ rows then $\bS_{\lambda}(U)=0$ and so $W_{\lambda}=0$.
The space $W \cap \Sym^k(U \otimes \C^d)$ is equal to $\bigoplus W_{\lambda} \otimes \bS_{\lambda}(\C^d)$.  Since
$\bS_{\lambda}(\C^d)$ generates $\bS_{\lambda}(\C^{\infty})$ whenever $\lambda$ has at most $d$ rows, we find that
$W \cap \Sym^k(U \otimes \C^d)$ generates $W$.
\end{proof}

Any twisted commutative algebra finitely generated in order 1 has the property that the partitions appearing in it
have a bounded number of rows.  Most of the results we prove about algebras generated in order 1, such as the above
theorem, can easily be extended to the larger class of algebras which have a bounded number of rows.  Going beyond this
class of algebras is harder though.  For instance, we do not know how much the order 1 hypothesis in
Theorem~\ref{tc:noeth} can be relaxed; see \S \ref{s:ques}.  We need one more general finiteness result about twisted
commutative algebras.

\begin{proposition}
\label{tc:inv}
Let $A$ be a noetherian twisted commutative algebra on which a reductive group $G$ acts.  Then $A^G$ is noetherian.  If
$M$ is a finitely generated $A$-module with a compatible action of $G$ then $M^G$ is a finitely generated $A^G$-module.
\end{proposition}

\begin{proof}
The usual proof works.  To show how it carries over, we prove that $A^G$ is noetherian.  We must therefore show
that $A^G \otimes F$ is a noetherian $A^G$-module for any finite length object $F$ of $\Sym(\Vec)$.  We have
an inclusion $A^G \otimes F \subset A \otimes F$.  Let $G$ act on $A \otimes F$ by acting trivially on $F$.  Let
$M$ be an $A^G$-submodule of $A^G \otimes F$ and let $M'$ be the $A$-submodule of $A \otimes F$ it generates.  Let
$y$ be an element of $M'$ which is $G$-invariant.  We can then write $y=\sum a_i x_i$ where the $a_i$ are elements of
$A$ and the $x_i$ are elements of $M$.  As $y$ and the $x_i$ are invariant, we have $y=\avg(y)=\sum \avg(a_i) x_i$.
Thus $y$ belongs to $M$.  This shows that $M=(M')^G$.  Therefore, the map
\begin{displaymath}
\{ \textrm{$A^G$-submodules of $A^G \otimes F$} \} \to \{ \textrm{$A$-submodules of $A \otimes F$} \}
\end{displaymath}
which takes an $A^G$-submodule to the $A$-submodule it generates is inclusion preserving and injective.  As any
ascending chain on the right side stabilizes, so too does any ascending chain on the left side.  This proves that
$A^G$ is noetherian.  We leave the statement about modules to the reader.
\end{proof}

\subsection{The category $\Sym(\mc{S})$}
\label{ss:schursym}

We now investigate the category $\Sym(\mc{S})$, which will feature prominently in what follows.  We begin by giving a
useful description of it.  Let $\Vec^f$ denote the category of finite families of vector spaces.  An object of $\Vec^f$
is a pair $(V, L)$ where $L$ is a finite set and $V$ assigns to each element $x$ of $L$ a vector space $V_x$.  A
morphism $(V, L) \to (V', L')$ consists of a bijection $i:L \to L'$ and a linear map $V_x \to V'_{i(x)}$ for each
$x \in L$.  We now claim that the following three categories are equivalent:
\begin{enumerate}
\item The category $\Sym(\mc{S})$.
\item The category of nice functors $\Vec^f \to \Vec$.
\item The category of sequences $(F_n)$ where $F_n$ is a nice $S_n$-equivariant functor $\Vec^n \to \Vec$.
\end{enumerate}
The equivalences, as well as the definition of ``nice'' in (b), should be clear.  Each of these categories is an
abelian tensor category and the various equivalences respect this structure.  The tensor structure in (b) is given
by
\begin{displaymath}
(F \otimes G)(V, L)=\bigoplus_{L=A \amalg B} F(V \vert_A, A) \otimes G(V \vert_B, B).
\end{displaymath}
The tensor structure in (c) is like that for the sequence model of $\Sym(\Vec)$.  Note that since $\mc{S}$ is itself a
tensor category we have a point-wise tensor product $\boxtimes$ on $\Sym(\mc{S})$.  In (b) this product is given by
\begin{displaymath}
(F \boxtimes G)(V, L)=F(V, L) \otimes G(V, L),
\end{displaymath}
hence the name ``point-wise tensor product.''  This tensor product will not be used in the rest of this section
but will show up later.

As $\Sym(\mc{S})$ is an abelian tensor category, we have the notions of algebras and modules over algebras in
$\Sym(\mc{S})$.  We note that for $A \in \Sym(\mc{S})$ giving a multiplication map $A \otimes A \to A$ amounts to
giving a natural map
\begin{displaymath}
A(V, L) \otimes A(V', L') \to A(V \amalg V', L \amalg L'),
\end{displaymath}
where $V \amalg V'$ denotes the natural map $L' \amalg L \to \Vec$ built out of $V$ and $V'$.  Define an \emph{element}
of $F \in \Sym(\mc{S})$ to be an element of $F(V, L)$ for some $(V, L)$.  We call $\# L$ the \emph{order} of
the element.  (We will use the term ``degree'' later to reference the grading on $\mc{S}$.)  As in the twisted
commutative setting, one can then give an elemental definition for ``finitely generated'' and this agrees with the
more general definition in terms of finite length objects given in \S \ref{ss:symcat}.

We now give examples of some of the above definitions to give the reader some sense of their nature.  Let $F_{\lambda}$
be the object of $\Sym(\mc{S})$ which assigns to $(V, L)$ the space 0 if $\# L \ne 1$ and the space
$\bS_{\lambda}(V_x)$ if $L=\{x\}$.  Then $F_{\lambda} \otimes F_{\mu}$ assigns to $(V, L)$ the space 0 if $\# L \ne 2$
and the space
\begin{displaymath}
\left[ \bS_{\lambda}(V_x) \otimes \bS_{\mu}(V_y) \right] \oplus
\left[ \bS_{\lambda}(V_y) \otimes \bS_{\mu}(V_x) \right]
\end{displaymath}
if $L=\{x,y\}$.  As a second example, $\Sym(F_{\lambda})$ is the object of $\Sym(\mc{S})$ which assigns to $(V, L)$
the space $\bigotimes_{x \in L} \bS_{\lambda}(V_x)$.  Note in particular that the only partition appearing in the
symmetric algebra $\Sym(F_{\lambda})$ is $\lambda$ itself.  These examples underline the fact that the product in
$\Sym(\mc{S})$ is formal:  the Littlewood--Richardson rule does intervene in any way.  (The Littlewood--Richardson
rule is used in the point-wise product $\boxtimes$ for $\Sym(\mc{S})$.)

The following is the main result we need on algebras in $\Sym(\mc{S})$.

\begin{theorem}
\label{wnnoeth}
An algebra in $\Sym(\mc{S})$ which is finitely generated in order 1 is noetherian.
\end{theorem}

\begin{proof}
Let $A$ be an algebra finitely generated in order 1.  Let $F$ be a finite length object of $\Sym(\mc{S})$.
We must show that $A \otimes F$ is a noetherian $A$-module.  As with any finitely generated algebra in $\Sym(\mc{S})$,
the number of rows in any partition appearing in $A$ is bounded; the same holds for $A \otimes F$.  Let $d$ be
an integer such that any partition appearing in $A \otimes F$ has at most $d$ rows, and let $U$ be a vector space
of dimension $d$.  For a finite set $L$ let $U_L$ be the constant family on $U$, i.e., the object $(V, L)$ of
$\Vec^f$ with $V_x=U$ for all $x$.  Then $L \mapsto U_L$ defines a functor $i:\Cfs \to \Vec^f$, which respects the
monoidal structure (disjoint union) on each category.  The induced functor $i^*:\Sym(\mc{S}) \to \Sym(\Vec)$ is a tensor
functor.  One easily sees that if $M'$ and $M$ are two sub-objects of $A \otimes F$ then $M=M'$ if and only
if $i^* M=i^* M'$.  We have thus shown that the map
\begin{displaymath}
\{ \textrm{$A$-submodules of $A \otimes F$} \} \to \{ \textrm{$(i^*A)$-submodules of $i^*(A \otimes F)$} \}
\end{displaymath}
is injective; it is obviously inclusion preserving.  Now, $i^* A$ is a twisted commutative algebra finitely generated
in order 1; to see this, write $A$ as a quotient of $\Sym(S)$ with $S \in \mc{S}$, so that $i^* A$ is a quotient of
$\Sym(S(U)\langle 1 \rangle)$.  We thus find that $i^* A$ is noetherian by Theorem~\ref{tc:noeth}.  As
$i^*(A \otimes F)=(i^* A) \otimes (i^* F)$ is a finitely generated $(i^* A)$-module, it is noetherian.  We thus find
that every ascending chain of $A$-submodules of $A \otimes F$ stabilizes, and so $A$ is noetherian.
\end{proof}

The following proposition is proved just like Proposition~\ref{tc:inv}.

\begin{proposition}
\label{wninv}
Let $A$ be a noetherian algebra in $\Sym(\mc{S})$ on which a reductive group $G$ acts.  Then $A^G$ is again
noetherian.  If $M$ is a finitely generated $A$-module with a compatible action of $G$ then $M^G$ is a finitely
generated $A^G$-module.
\end{proposition}

\subsection{$\Delta$-modules}

Let $\Vec^{\Delta}$ be the category whose objects are pairs $(V, L)$ as in $\Vec^f$, but where now a morphism $(V, L)
\to (V', L')$ consists of a surjection $L' \to L$ together with a map $V_x \to \bigotimes_{y \mapsto x} V'_y$ for each
$x \in L$.  We now come to a central concept in this paper:

\begin{definition}
A \emph{$\Delta$-module} is a nice functor $\Vec^{\Delta} \to \Vec$.
\end{definition}

In the above definition, we say that a functor $\Vec^{\Delta} \to \Vec$ is nice if it is so when
restricted to $\Vec^f$.  We denote the category of these functors by $\Mod_{\Delta}$.  It is abelian, though not
semi-simple.  Since $\Vec^f$ is a sub-category of $\Vec^{\Delta}$, every
$\Delta$-module defines an object of $\Sym(\mc{S})$.  For a $\Delta$-module $F$ we let $[F]$ be the class in
$\Sym(K(\mc{S}))$ obtained by regarding $F$ as an object of $\Sym(\mc{S})$.

The forgetful functor $\Mod_{\Delta} \to \Sym(\mc{S})$ has a left adjoint, which we denote by $\Phi$.  To
describe this functor explicitly, we must introduce a piece of notation.  Let $(V, L)$ be an object of $\Vec^f$
and let $\ms{U}$ be a partition of $L$.  For a set $S \in \ms{U}$ put $V_S=\bigotimes_{x \in S} V_x$.  Then
$(V, \ms{U})$ is an object of $\Vec^f$.  There is a natural map $(V, \ms{U}) \to (V, L)$ in $\Vec^{\Delta}$, the
surjection $L \to \ms{U}$ mapping an element of $L$ to the part of $\ms{U}$ to which it belongs.  With this notation
in hand, we can give the following formula for $\Phi$:
\begin{displaymath}
(\Phi F)(V, L)=\bigoplus F(V, \ms{U}),
\end{displaymath}
where the sum is over all partitions $\ms{U}$ of $L$.  We leave it to the reader to work out the $\Delta$-structure on
$\Phi(F)$ and verify its universal property.  We call $\Phi(F)$ the \emph{free $\Delta$-module} on $F$.  By a
\emph{(finite) free} $\Delta$-module we mean one isomorphic to $\Phi(F)$, where $F$ is a (finite length) object in
$\Sym(\mc{S})$.  Free modules are projective since $\Phi$ is a left adjoint and $\Sym(\mc{S})$ is semi-simple.

We define an \emph{element} of a $\Delta$-module $F$ to be an element
of $F(V, L)$ for some $(V, L)$.  Given a collection of elements of $F$ one can speak of the $\Delta$-submodule that
it generates.  We say that $F$ is \emph{finitely generated} if there is a finite set of elements of $F$ that generate
it.  This is equivalent to $F$ being a quotient of a finite free $\Delta$-module.  We say that a $\Delta$-module
is \emph{noetherian} if every ascending chain of $\Delta$-submodules stabilizes.  Note that noetherian implies
finitely generated.

Let $n$ be a positive integer.  Let $T_n$ be the object of $\Sym(\mc{S})$ which assigns to $(V, L)$ the space
$V_x^{\otimes n}$ if $L=\{x\}$ and 0 otherwise.  Let $W_n$ be the symmetric algebra on $T_n$ in the category
$\Sym(\mc{S})$.  It is given by
\begin{displaymath}
W_n(V, L)=\bigotimes_{x \in L} V_x^{\otimes n}
\end{displaymath}
The multiplication map in $W_n$ is given by concatenation of tensors.  Note that $S_n$ acts on $W_n$ by algebra
homomorphisms.  We also consider the free $\Delta$-module on $T_n$.  It is given by
\begin{displaymath}
(\Phi T_n)(V, L)=\bigotimes_{x \in L} V_x^{\otimes n}
\end{displaymath}
if $\# L \ge 1$, while $(\Phi T_n)(V, L)=0$ for $\#L=0$.  Observe that $\Phi(T_n)$ is naturally a
subobject of $W_n$ in the category $\Sym(\mc{S})$, and is in fact an ideal.  The following is the key result connecting
$\Delta$-modules to objects we have previously studied, the final rung of the ladder of \S \ref{ss:outline}.

\begin{proposition}
\label{delta-w}
Any $\Delta$-submodule of $\Phi(T_n)$ is a $W_n^{S_n}$-submodule of $\Phi(T_n)$.
\end{proposition}

\begin{proof}
Let $(V, L)$ and $(V', L')$ be two objects of $\Vec^f$.  Let $W=\bigotimes_{y \in L'} V'_y$, so that
$W_n(V', L')=W^{\otimes n}$.  Let $v \in (\Phi T_n)(V, L)$, and let $v' \in W^{\otimes n}$ be $S_n$-invariant.  We must
show that the image of $v \otimes v'$ under the multiplication map $(\Phi T_n)(V, L) \otimes W^{\otimes n} \to
(\Phi T_n)(V \amalg V', L \amalg L')$ belongs to the $\Delta$-submodule of $\Phi T_n$ generated by $v$.  Now,
$(W^{\otimes n})^{S_n}$ is spanned by $n$th tensor powers of elements of $W$.  It thus suffices to treat the case
$v'=w^{\otimes n}$ for some $w \in W$.

Pick an element $x_0$ of $L$.  Define a map $f:(V, L) \to (V \amalg V', L \amalg L')$ in $\Vec^{\Delta}$, as follows.
The map $L \amalg L' \to L$ is the identity on $L$ and collapses all of $L'$ to $x_0$.  For $x \ne x_0$, the map
$f_x:V_x \to V_x$ is the identity.  The map $f_{x_0}:V_{x_0} \to V_{x_0} \otimes W$ is given by $\id \otimes w$.
The map $f$ induces a map $(\Phi T_n)(V, L) \to (\Phi T_n)(V \amalg V', L \amalg L')$, which one easily verifies
is the map induced by multiplication by $w^{\otimes n}$ on $W_n$.  Thus if $v$ is an element of $(\Phi T_n)(V, L)$, then
its product with $w^{\otimes n}$ in $W_n$ can be computed by taking its image under $(\Phi T_n)(f)$.
This shows that the product of $v$ and $v'$ belongs to the $\Delta$-module generated by $v$, which completes
the proof.
\end{proof}

\begin{theorem}
\label{tnnoeth}
The $\Delta$-module $\Phi(T_n)$ is noetherian.
\end{theorem}

\begin{proof}
The algebra $W_n$ is noetherian by Theorem~\ref{wnnoeth}, and so $W_n^{S_n}$ is noetherian by
Proposition~\ref{wninv}.  As $W_n$ is a finite $W_n^{S_n}$-module, it is a noetherian $W_n^{S_n}$-module.  The
same holds for the submodule $\Phi(T_n)$.  If $M_i$ is an ascending chain of $\Delta$-submodules of $\Phi(T_n)$
then it is an ascending chain of $W_n^{S_n}$-submodules by the previous proposition, and therefore stabilizes.
Thus $\Phi(T_n)$ is a noetherian $\Delta$-module.
\end{proof}

Call a $\Delta$-module \emph{small} if it is a subquotient of a finite direct sum of $\Phi(T_n)$'s (with
$n$ allowed to vary).  The above theorem implies that small $\Delta$-modules are noetherian, and in particular
finitely generated.  We record the following result, which follows immediately from the definitions and
Proposition~\ref{delta-w}.

\begin{proposition}
\label{delta-filt}
Let $F$ be a small $\Delta$-module.  Then there exists a finite chain $0=F_0 \subset \cdots
\subset F_r=F$ of $\Delta$-submodules of $F$ and integers $n_i$ such that $F_i/F_{i-1}$, regarded as an object
of $\Sym(\mc{S})$, can be given the structure of a finitely generated module over $W_{n_i}^{S_{n_i}}$.
\end{proposition}

\begin{remark}
\label{delta-rem}
We can in fact show that all finitely generated $\Delta$-modules are noetherian.  The argument in the general case
is by a Gr\"obner degeneration, and is much different than our argument for $\Phi(T_n)$ presented above.  However, the
above argument for $\Phi(T_n)$, which relates $\Delta$-submodules to modules over $W_n^{S_n}$, is important for our
later arguments with Hilbert series.
\end{remark}

\subsection{More on $\Delta$-modules}
\label{ss:moredelta}

Let $F$ be a $\Delta$-module.  We define $F^{\rm{old}}(V, L)$ to be the space spanned by the images of the maps
$F(V, \ms{U}) \to F(V, L)$, as $\ms{U}$ varies over all non-discrete partitions of $L$.  One easily verifies that
$F^{\rm{old}}$ is a $\Delta$-submodule of $F$.  We define a functor
\begin{displaymath}
\Psi:\Mod_{\Delta} \to \Sym(\mc{S}), \qquad \Psi(F)=F/F^{\rm{old}}.
\end{displaymath}
Note that $\Psi(F)$ is naturally a $\Delta$-module.  However, if $(V, L) \to (V', L')$ is a map in $\Vec^{\Delta}$ and
$L' \to L$ is not an isomorphism, then $\Psi(F)$ applied to this map is zero; this is why we regard $\Psi(F)$ as an
object of $\Sym(\mc{S})$.  In fact, $\Psi(F)$ is the universal quotient of $F$ with this property.  One may thus
regard $\Psi(F)$ as the maximal semi-simple quotient (i.e., cosocle) of $F$.

A $\Delta$-module $M$ is finitely generated if and only if $\Psi(M)$ is a finite length object of $\Sym(\mc{S})$; this
is a version of Nakayama's lemma.  In fact, $M$ is a quotient of $\Phi(\Psi(M))$, though non-canonically.  We have
$\Psi(\Phi(F))=F$.  The functor $\Psi$ is right exact, but not exact.  Its left derived functors $L^i \Psi$
exist.  If $M$ is a finitely generated $\Delta$-module then $L^i \Psi(M)$ is a finite length object of $\Sym(\mc{S})$;
this can be deduced easily from the fact that finitely generated $\Delta$-modules are noetherian.  One can recover
$[M]$ from $[L \Psi M]$ by applying $\Phi$.  Thus the sequence of polynomials $[L^i \Psi M]$ contains
more information than the series $[M]$.

We now give an alternative definition of $\Delta$-modules.  Recall that for an object $(V, L)$ of $\Vec^f$ and
a partition $\ms{U}$ of $L$, there is a natural map $(V, \ms{U}) \to (V, L)$ in $\Vec^{\Delta}$.  One easily
verifies that every map in $\Vec^{\Delta}$ can be factored as one of these maps followed by a map in $\Vec^f$.
In fact, we can even say a bit more.  Call a partition \emph{little} if all its parts are
singletons, except one which has two elements.  Call a map $(V, \ms{U}) \to (V, L)$ \emph{little} if $\ms{U}$ is.
One then verifies that any map $(V, \ms{U}) \to (V, L)$ can be factored into a sequence of little maps.  Thus
all morphisms in $\Vec^{\Delta}$ can be factored into little maps and maps in $\Vec^f$.  Therefore, a
$\Delta$-module can be thought of as an object of $\Sym(\mc{S})$ together with the extra data of functoriality
with respect to little maps.

This extra data can be recorded in an elegant manner.  Let $m^*:\mc{S} \to \mc{S}^{\otimes 2}$ be the
co-multiplication map.  It takes $F \in \mc{S}$ to the functor $m^* F \in \mc{S}^{\otimes 2}$ given by $(V, W) \mapsto
F(V \otimes W)$.  The functor $m^* F$ has a natural $S_2$-equivariant structure and so defines an object of
$\Sym^2(\mc{S})$.  There is a unique extension of $m^*$ to a derivation
\begin{displaymath}
\Delta:\Sym(\mc{S}) \to \Sym(\mc{S}).
\end{displaymath}
Here by ``derivation'' we mean $\Delta$ satisfies the Leibniz rule and interacts correctly with divided powers
(Schur functors).  A $\Delta$-module is then just an object $M$ of $\Sym(\mc{S})$ together with a map
$\Delta M \to M$ which satisfies a certain associativity condition, which we do not write out.  The map
$\Delta M \to M$ precisely records the functoriality of $M$ with respect to little maps in $\Vec^{\Delta}$, and
the associativity condition ensures that $M$ extends to a functor with respect to all maps in $\Vec^{\Delta}$.
The image of the map $\Delta M \to M$ is $M^{\rm{old}}$, and so its cokernel is $\Psi(M)$.

There is an analogy between $\Delta$-modules and graded $\C[t]$-modules.  The category $\Sym(\mc{S})$ is
analogous to the category of graded vector spaces.  The functor $\Phi$ is analogous to the functor which takes a
graded vector space $V$ to the free graded $\C[t]$-module $\C[t] \otimes V$, while the functor $\Psi$ is analogous to
the functor which takes a graded $\C[t]$-module $M$ to the graded vector space $M \otimes_{\C[t]} \C$.  The map
$\Delta M \to M$ is analogous to multiplication by $t$, while the space $M^{\rm{old}}$ is analogous to the image
of $t$.  One might hope that $M \mapsto L\Psi(M)$ provides an equivalence between the derived category
of $\Mod_{\Delta}$ and some other natural derived category, in analogy with Koszul duality; we have not worked
this out.

\section{Hilbert series}
\label{s:hilbert}

In this section we develop the theory of Hilbert series for certain objects of $\Sym(\Vec)$, $\Sym(\mc{S})$ and
$\Mod_{\Delta}$.  The main results are rationality theorems.  If $A$ is a finitely generated graded ring, in the usual
sense, one can prove the rationality of its Hilbert series by picking a surjection $P \to A$, where $P$ is a polynomial
ring, resolving $A$ by free $P$-modules and then explicitly computing the Hilbert series of a free $P$-module.  The key
fact that makes this work is that $P$ has finite global dimension.  In the setting of twisted commutative algebras,
this approach is no longer viable:  the twisted commutative algebra $\Sym(U \langle 1 \rangle)$ has infinite global
dimension for any non-zero $U$.  The reason for this is that no wedge power of $U\langle 1 \rangle$ vanishes, so the
Koszul complex does not terminate!  We get around this problem by relating Hilbert series of twisted commutative
algebras to equivariant Hilbert series of usual rings, where we can use the usual methods.  To study Hilbert series
of objects in $\Sym(\mc{S})$, we relate them to Hilbert series of twisted commutative algebras.  Finally, to study
Hilbert series of objects in $\Mod_{\Delta}$ (what we ultimately care about), we relate them to Hilbert series of
objects in $\Sym(\mc{S})$.

\subsection{Hilbert series in $\Sym(\Vec)$}

Let $M$ be an object of $\Sym(\Vec)$, taken in the sequence model.  We assume each $M_n$ is finite dimensional.  We
define the \emph{Hilbert series} $H_M$ of $M$ by:
\begin{displaymath}
H_M(t)=\sum_{n=0}^{\infty} \frac{1}{n!} (\dim{M_n}) \, t^n.
\end{displaymath}
Of course, $H_M(t)$ is simply the element $[M]$ of $\Sym(K(\Vec))=\Q \lbb t \rbb$.  Our goal in this section is to
demonstrate the following theorem:

\begin{theorem}
\label{hilbert}
Let $A$ be a twisted commutative algebra finitely generated in order 1 and let $M$ be a finitely generated
$A$-module.  Then $H_M(t)$ is a polynomial in $t$ and $e^t$.
\end{theorem}

Define $H^*_M(t)$ similarly to $H_M(t)$, but without the factorial. The theorem is equivalent to the following
statement, which is what we actually prove:  we have
\begin{displaymath}
H_M^*(t)=\sum_{k=0}^d \frac{p_k(t)}{(1-kt)^{a_k}}
\end{displaymath}
for some integer $d$, polynomials $p_k(t)$ and non-negative integers $a_k$.  Note that for a module over a graded ring,
in the usual sense, the Hilbert series only has a pole at $t=1$, while our Hilbert series for modules over twisted
commutative algebras can have poles at $t=1/k$ for any non-negative integer $k$.

The above theorem only applies to modules over twisted commutative algebras generated in order 1, and is false
more generally.  For example, let $M=A=\Sym((\C^{\infty})^{\otimes 2})$, a twisted commutative algebra generated in
order 2.  Then $H_M(t)=e^{t^2}$.  Although this is not a polynomial in $t$ and $e^t$, it is a very reasonable function,
and one can hope that there is a nice generalization of Theorem~\ref{hilbert}.

Before getting into the proof of Theorem~\ref{hilbert} we introduce equivariant Hilbert series.  Say a group $G$
acts on an object $M$ of $\Sym(\Vec)$.  We define its \emph{$G$-equivariant Hilbert series} $H^*_{M, G}$ by:
\begin{displaymath}
H^*_{M, G}(t)=\sum_{n=0}^{\infty} [M_n] t^n
\end{displaymath}
where $[M_n]$ denotes the class of $M_n$ in the Grothendieck group $K(G)$.  Thus $H^*_{M, G}$ is a power series with
coefficients in the ring $K(G)$.  We will need to use these Hilbert series in our proof of Theorem~\ref{hilbert}
and we will also need a generalization of Theorem~\ref{hilbert} to the equivariant setting.

We now begin the proof of Theorem~\ref{hilbert}.  Thus let $A$ and $M$ be given.  Since $A$ is finitely generated
in order one, it is a quotient of $\Sym(U\langle 1 \rangle)$ for some finite dimensional vector space $U$.  Of course,
$M$ is a finitely generated module over $\Sym(U\langle 1 \rangle)$.  It thus suffices to consider the case where
$A=\Sym(U\langle 1 \rangle)$.  Now, regard $A$ and $M$ in the Schur model.  If $\bS_{\lambda}$ occurs in $A$ then
$\lambda$ has at most $\dim{U}$
rows.  Since $M$ is finitely generated, it is a quotient of $A \otimes S$ for some finite length object $S$ of
$\mc{S}$; it follows that there is an integer $d$ such that only those $\bS_{\lambda}$ for which $\lambda$ has at
most $d$ rows appear in $M$.  We therefore do not lose information by considering $M(\C^d)$ with its $\GL(d)$ action.
In fact, we can even consider $M(\C^d)$ with its $T$ action without losing information, where $T$ is the diagonal
torus in $\GL(d)$.  The main idea of the proof of Theorem~\ref{hilbert} is to relate $H^*_M$ to $H^*_{M(\C^d), T}$,
prove that the latter is of a specific form and then deduce from this the rationality of $H^*_M$.  (One can regard any
graded $\C$-algebra as a twisted commutative algebra.  Thus $H^*_{M(\C^d), T}$ makes sense.  In fact, it agrees with
the usual $T$-equivariant Hilbert series of $M(\C^d)$.)

We need to introduce a bit of notation related to $T$.  We let $\alpha_1, \ldots, \alpha_d$ be the standard projections
$T \to \G_m$.  We
define an involution of the coordinate ring of $T$, denoted with an overline, by $\ol{\alpha}_i=\alpha_i^{-1}$, and we
write $\vert x \vert^2$ for $x \ol{x}$.  We let $\Delta(\alpha)$ be the discriminant $\prod_{i<j} (\alpha_i-\alpha_j)$.
For a character $\chi$ of $T$ we define $\int_T \chi(\alpha) d\alpha$ to be 1 if $\chi$ is trivial and 0 otherwise,
and we extend $\int_T d\alpha$ linearly to all functions on $T$.
(The symbol $\int_T d\alpha$ is just notation and does not indicate actual integration.)  If $\chi_1$ and $\chi_2$ are
characters of irreducible representations of $\GL(d)$ then Weyl's integration formula (see \cite[\S 26.2]{FH}), stated
in our language, reads
\begin{displaymath}
\frac{1}{d!} \int_T \chi_1(\alpha) \chi_2(\ol{\alpha}) \vert \Delta(\alpha) \vert^2 d\alpha=\begin{cases}
1 & \textrm{if $\chi_1=\chi_2$} \\
0 & \textrm{if $\chi_1 \ne \chi_2$}
\end{cases}
\end{displaymath}
We identify $K(T)$ with $\Q[\alpha_i, \alpha_i^{-1}]$ so that a $T$-equivariant Hilbert series can be identified with
a power series in $t$ whose coefficients are Laurent polynomials in the $\alpha_i$.  The following is the key step in
our understanding of $H^*_M$:

\begin{lemma}
\label{hilbertlem}
We have
\begin{displaymath}
H^*_M(t)=\frac{1}{d!} \int_T H^*_{M(\C^d), T}(t; \alpha) \frac{\vert \Delta(\alpha) \vert^2}{1-\sum \ol{\alpha}_i}
d\alpha.
\end{displaymath}
\end{lemma}

\begin{proof}
Write $M=\bigoplus \bS_{\lambda}^{\oplus m_{\lambda}}$, the sum taken over $\lambda$.  We then have:
\begin{displaymath}
H^*_M(t)=\sum_{\lambda} m_{\lambda} \cdot (\dim \Sp_{\lambda}) \cdot t^{\vert \lambda \vert}.
\end{displaymath}
On the other hand
\begin{displaymath}
H^*_{M(\C^d), T}(t; \alpha)=\sum_{\lambda} m_{\lambda} \cdot \textrm{(the character of $\bS_{\lambda}(\C^d)$)}
\cdot t^{\vert \lambda \vert}.
\end{displaymath}
Put
\begin{displaymath}
f(\alpha)=\sum_{\lambda} \textrm{(the character of $\bS_{\lambda}(\C^d)$)} \cdot \dim{\Sp_{\lambda}}.
\end{displaymath}
Weyl's integration formula now gives us
\begin{displaymath}
H^*_M(t)=\frac{1}{d!} \int_{T'} H^*_{M(\C^d), T}(t; \alpha) f(\ol{\alpha}) \vert \Delta(\alpha) \vert^2 d\alpha.
\end{displaymath}
We must compute $f(\alpha)$.  Observe:
\begin{displaymath}
\bigoplus_{k=0}^{\infty} (\C^d)^{\otimes k}=\bigoplus_{\lambda} \bS_{\lambda}(\C^d) \otimes \Sp_{\lambda}.
\end{displaymath}
The character of the right side is $f(\alpha)$.  The character of the left side is
\begin{displaymath}
\sum_{k=0}^{\infty} \left( \sum_{i=1}^d \alpha_i \right)^k=\frac{1}{1-\sum \alpha_i}.
\end{displaymath}
This yields the stated formula.
\end{proof}

We have thus related $H^*_M$, what we care about, to $H^*_{M(\C^d), T}$, which should be easier to understand since
$M(\C^d)$ is a finitely generated module over the polynomial ring $A(\C^d)$.  We now see that $H^*_{M(\C^d), T}$ is
indeed easy to understand:

\begin{lemma}
We have
\begin{displaymath}
H^*_{M(\C^d), T}(t; \alpha)=\frac{p(t; \alpha)}{\prod_{i=1}^d (1-\alpha_i t)^n}
\end{displaymath}
where $p(t; \alpha)$ is a polynomial and $n=\dim{U}$.
\end{lemma}

\begin{proof}
The terms of the minimal resolution for $M(\C^d)$ over $A(\C^d)$ are $A(\C^d) \otimes E_{\bullet}$ where
\begin{displaymath}
E_i=\Tor_i^{A(\C^d)}(M(\C^d), \C).
\end{displaymath}
Since $\Tor$ is functorial, each $E_i$ carries an action of $\GL(d)$ (and
therefore $T$), and the minimal resolution is equivariant (or rather, can be taken to be so).  Thus the $T$-equivariant
Hilbert series for $M$ is the alternating sum of those for $A(\C^d) \otimes E_i$; of course, each of these is the
product of those for $A(\C^d)$ and $E_i$.  Since $E_i$ is a finite dimensional representation of $T$ its Hilbert series
is a polynomial.  Thus the lemma is reduced to the case $M=A$. Now,
\begin{displaymath}
A(\C^d)=\Sym(U \otimes \C^d)=\Sym(\C^d \oplus \cdots \oplus \C^d)=\Sym(\C^d) \otimes \cdots \otimes \Sym(\C^d)
\end{displaymath}
where $\C^d$ is summed with itself $n=\dim{U}$ times.  We thus find that $H_{A(\C^d), T}$ is the $n$th power of
$H^*_{\Sym(\C^d), T}$.  Similarly, $\Sym(\C^d)=\Sym(\C \oplus \cdots \oplus \C)$, where there are $d$ copies of $\C$ and
$T$ acts on the $i$th one by the character $\alpha_i$.  Thus $H^*_{\Sym(\C^d), T}
=\prod (1-\alpha_i t)^{-1}$.  This proves the lemma.
\end{proof}

Combining the two lemmas, we obtain an expression
\begin{displaymath}
H^*_M(t)=\int_T \frac{p(t; \alpha)}{\prod (1-\alpha_i t)^n} \frac{1}{1-\sum \ol{\alpha}_i} d\alpha
\end{displaymath}
where $p(t; \alpha)$ is a polynomial in $t$, the $\alpha_i$ and the $\alpha_i^{-1}$.  (We have absorbed the $1/d!$ and
$\vert \Delta(\alpha) \vert^2$ into $p$.)  Expanding the integrand into a power series, we find
\begin{displaymath}
H^*_M(t)=\int_T \left[ \sum_{k, \ell} \qbinom{k}{n}{} \alpha^k \left( \sum \ol{\alpha_i} \right)^{\ell} p(t; \alpha)
t^{\vert k \vert} \right] d \alpha
\end{displaymath}
where the sum is taken over $k \in \Z_{\ge 0}^d$ and $\ell \in \Z_{\ge 0}$.  Here
\begin{displaymath}
\qbinom{k}{n}{}=\binom{k_1+n-1}{n-1} \cdots \binom{k_d+n-1}{n-1}, \qquad \alpha^k=\alpha_1^{k_1} \cdots \alpha_d^{k_d}
\qquad \textrm{and} \qquad \vert k \vert=k_1+\cdots+k_d.
\end{displaymath}
We must show that this is a rational function in $t$.  It suffices, by linearity, to treat the case where $p(t;
\alpha)=t^{e_0} \alpha_1^{e_1} \cdots \alpha_d^{e_d}$ where the $e_i$ are integers.  Of course, the $t^{e_0}$
factor does not really affect anything, so we leave it out.  We are thus reduced to showing that
\begin{displaymath}
\int_T \left[ \sum_{k, \ell} \qbinom{k}{n}{} \alpha^{k+e} \left( \sum \ol{\alpha_i} \right)^{\ell} t^{\vert k \vert}
\right] d\alpha
\end{displaymath}
is rational.  By degree considerations, the $(k, \ell)$ term in the above sum integrates to zero unless $\ell=\vert k
+e \vert$.  Furthermore, when $\ell=\vert k+e \vert$ only one monomial in $( \sum \ol{\alpha_i} )^{\ell}$ contributes
a non-zero quantity, namely the one where $\ol{\alpha}_i$ has exponent $k_i+e_i$.  We therefore find that the above is
equal to
\begin{displaymath}
\sum_k \qbinom{k}{n}{} C_{k+e} t^{\vert k \vert}
\end{displaymath}
where
\begin{displaymath}
C_k=\frac{(\vert k \vert)!}{(k_1)! \cdots (k_d)!}
\end{displaymath}
is the multinomial coefficient.  (We use the convention that $C_k=0$ if any of the coordinates of $k$ are negative.)
Theorem~\ref{hilbert} now follows from the following lemma:

\begin{lemma}
\label{sumlem}
Let $d$ be a positive integer, let $e \in \Z^d$ and let $p$ be a polynomial of $d$ variables.  Then
\begin{displaymath}
\sum p(k) C_{k+e} t^{\vert k \vert}
\end{displaymath}
is a rational function of $t$, with poles only at $t=1/a$ where $1 \le a \le d$ is an integer.  (The sum is taken over
$k \in \Z_{\ge 0}^d$.)
\end{lemma}

\begin{proof}
Observe that the formula
\begin{displaymath}
k_1 C_{k_1, k_2, \ldots, k_d}=\vert k \vert C_{k_1-1, k_2, \ldots, k_d}
\end{displaymath}
is valid for any tuple of integers $k \in \Z^d$.  To prove the lemma, it suffices to treat the case where $p$ is a
monomial.  Thus, if $p$ is not a constant, we can write $p=k_i p'$ for some index $i$ and some smaller monomial $p'$.
We therefore have
\begin{align*}
\sum p(k) C_{k+e} t^{\vert k \vert}
&=\sum k_i p'(k) C_{k+e} t^{\vert k \vert} \\
&=\sum p'(k) (k_i+e_i-e_i) C_{k+e} t^{\vert k \vert} \\
&=\sum p'(k) \vert k+e \vert C_{k+e'} t^{\vert k \vert} - e_i \sum p'(k) C_{k+e} t^{\vert k \vert}
\end{align*}
where $e'$ is obtained from $e$ by replacing $e_i$ with $e_i-1$.  In the right term we have replaced $p$ with a lower
degree polynomial.  In the left term we have replaced $p$ with an equal degree polynomial, but one that is of the
form $p'(k) \vert k \vert$ where $p'$ has smaller degree.  (Note that $\vert k + e \vert=\vert k \vert +
\vert e \vert$.)  It follows that by repeatedly applying the above process, we can reduce to the case where $p$ is
a function of $\vert k \vert$.  Now, note that
\begin{displaymath}
\sum \vert k \vert^n C_{k+e} t^{\vert k \vert}=\left( t \frac{d}{dt} \right)^n \sum C_{k+e} t^{\vert k \vert}.
\end{displaymath}
It thus suffices to show that $\sum C_{k+e} t^{\vert k \vert}$ is a rational function.  We have
\begin{displaymath}
\sum_{k \ge 0} C_{k+e} t^{\vert k \vert}=\sum_{k \ge e} C_k t^{\vert k-e \vert}=t^{-\vert e \vert} \sum_{k \ge e}
C_k t^{\vert k \vert}
\end{displaymath}
where $k \ge e$ means $k_i \ge e_i$ for each $i$.  Now, the terms in the right sum for which some $k_i$ is negative
are zero and therefore do not contribute.  We can thus assume that each $e_i$ is non-negative.  We can also ignore
the $t^{-\vert e \vert}$ factor.  Now, write $k=(k_1, k')$ where $k'$ is a $d-1$ tuple, and do similarly for
$e$.  Then
\begin{displaymath}
\sum_{k \ge e} C_k t^{\vert k \vert}
=\sum_{k_1 \ge e_1, k' \ge e'} C_k t^{\vert k \vert}
=\sum_{k_1 \ge 0, k' \ge e'} C_k t^{\vert k \vert}
-\sum_{k_1=0}^{e_1-1} \sum_{k' \ge e'} C_k t^{\vert k \vert}
\end{displaymath}
Now, for $k_1$ fixed, we have
\begin{displaymath}
\sum_{k' \ge e'} C_k t^{\vert k \vert}=\frac{t^{k_1}}{k_1!} \sum_{k' \ge e'} (\vert k' \vert + k_1) \cdots
(\vert k' \vert +1) C_{k'} t^{\vert k' \vert}.
\end{displaymath}
Thus each of the sums on the right in the previous expression is of the general form that we are considering in this
lemma but with a smaller $d$.  We can therefore assume that they are each rational by induction.  By repeating this
procedure, we can thus reduce to the case $e=0$.  The identity
\begin{displaymath}
\sum_{\vert k \vert=n} C_k=d^n
\end{displaymath}
now gives
\begin{displaymath}
\sum_{k \ge 0} C_k t^{\vert k \vert}=\frac{1}{1-dt}.
\end{displaymath}
This completes the proof.
\end{proof}

\subsection{Equivariant Hilbert series in $\Sym(\Vec)$}

Unfortunately, Theorem~\ref{hilbert} is too weak for our eventual applications.  Before stating the result we need,
we make a definition for the sake of clarity:

\begin{definition}
Let $A$ be a ring.  A series $f \in A \lbb t \rbb$ is \emph{rational} if there exists a polynomial $q \in A[t]$
with $q(0)=1$ such that $qf$ is a polynomial.
\end{definition}

Note that it could be that $f \in A \lbb t \rbb$ is not rational, but that there is an extension $A \subset B$ so
that $f$ is rational when regarded as an element of $B \lbb t \rbb$.  Thus a bit of care needs to be taken with the
definition.  However, we do have the following simple result, the proof of which is left to the reader:

\begin{lemma}
\label{lem:ratl}
Let $A \subset B$ be an inclusion of rings and let $f$ be an element of $A \lbb t \rbb$ such that $f$ is rational when
regarded as an element of $B \lbb t \rbb$.  Then $f$ itself is rational in the following cases:  (1) there is a finite
group $G$ acting on $B$ such that $A=B^G$; (2) $A$ and $B$ are fields.
\end{lemma}

The main result of this section is the following:

\begin{theorem}
\label{hilbert2}
Let $G$ be a connected reductive group, let $\Gamma$ be a finite group, let $A$ be a twisted commutative algebra
finitely generated in order 1 on which $G \times \Gamma$ acts and let $M$ be a finitely generated $A$-module with
a compatible action of $G \times \Gamma$.  Then the $G$-equivariant Hilbert series $H^*_{M^{\Gamma}, G}$ of
$M^{\Gamma}$, regarded as an element of the power series ring $K(G) \lbb t \rbb$, is a rational function.
\end{theorem}

Most likely, this proposition could be generalized by replacing $\Gamma \subset G \times \Gamma$ with an arbitrary
normal reductive subgroup of an arbitrary reductive group.  We do not need this more general result, and so only prove
the special one, which allows for some simplifications in the proof.  With some book-keeping, one can also show that
the denominator of $H^*_{M^{\Gamma}, G}$ has a particular form, but we do not do this.  The theorem implies a certain
result about $H_{M^{\Gamma}, G}$, but one that is not so elegant:  exponentials of algebraic functions (roots of
polynomials over $K(G)$) appear.  We prove this proposition following the same plan as the proof of last one, after
some preliminary reductions.

First, we observe that it suffices to prove that $H^*_{M, \Gamma \times G}$, an element of $K(\Gamma \times G) \lbb t
\rbb$, is a rational function.  To see this, assume we have an equation
$(1+tq)H^*_{M, \Gamma \times G}=p$ with $p$ and $q$ in $K(\Gamma \times G) \lbb t \rbb$.  Now, observe that $K(\Gamma)$
can be thought of as the ring of class functions on $\Gamma$ (at least, after an extension of scalars, which does
not affect rationality by Lemma~\ref{lem:ratl}) and so we may write
\begin{displaymath}
H^*_{M, \Gamma \times G}=\sum H_i \delta_i, \qquad q=\sum q_i \delta_i, \qquad p=\sum p_i \delta_i
\end{displaymath}
where $H_i$, $q_i$ and $t_i$ belong to $K(G) \lbb t \rbb$ and the $\delta_i$ are characteristic functions of
conjugacy classes in $\Gamma$.  Since the $\delta_i$ are orthogonal idempotents, we have
\begin{displaymath}
(1+tq) H^*_{M, \Gamma \times G}=\sum (1+tq_i) H_i \delta_i = \sum p_i \delta_i
\end{displaymath}
and so $(1+tq_i)H_i=p_i$ holds for each $i$.  Thus each $H_i$ is a rational function in $K(G) \lbb t \rbb$.  Since
$H^*_{M^{\Gamma}, G}=(\# \Gamma )^{-1} \sum H_i$, it follows that it too is a rational function.  This
establishes our claim.

Now, we can think of the coefficients of $H^*_{M, \Gamma \times G}$ as class functions on $\Gamma$.  The series
$H^*_{M, \Gamma \times G}$ defines a rational element of $K(\Gamma \times G) \lbb t \rbb$ if and only if the
series $H^*_{M, \gamma, G}$ obtained by evaluating on the element $\gamma \in \Gamma$ is a rational element of
$K(G) \lbb t \rbb$, for each $\gamma$.  Let $\Gamma'$ be the cyclic subgroup generated by some $\gamma \in \Gamma$.
By the same reasoning, $H^*_{M, \gamma, G}$ will be a rational element of $K(G) \lbb t \rbb$ if $H^*_{M, \Gamma'
\times G}$ is a rational element of $K(\Gamma' \times G) \lbb t \rbb$.  Thus it suffices to show that for each cyclic
subgroup $\Gamma'$ of $\Gamma$, the series $H^*_{M \Gamma' \times G}$ is a rational element of $K(\Gamma' \times G)
\lbb t \rbb$.  In other words, we may assume from the outset that $\Gamma$ is cyclic.

We make another reduction.  We have $K(G)=K(H)^W$ where $H$ is a maximal torus in $G$ and $W$ is its Weyl group.
By Lemma~\ref{lem:ratl}, a power series with coefficients in $K(G)$ is rational if and only if it is so when
regarded with $K(H)$ coefficients.  Thus it suffices to show that $H^*_{M, \Gamma \times H}$ is rational.  In other
words, we may as well assume from the outset that $G$ is a torus.

Finally, as in the proof of the Theorem~\ref{hilbert}, we may assume $A=\Sym(U\langle 1 \rangle)$, where $U$ is a
finite dimensional representation of $\Gamma \times G$.  Since $\Gamma \times G$ is a commutative reductive group, we
can write $U=\bigoplus_{j=1}^n \C \psi_j$ where the $\psi_j$ are characters of $\Gamma \times G$.  Pick $d$ large
compared to the number of rows appearing in $M$ and let $T \subset \GL(d)$ be the diagonal torus.  We now have:

\begin{lemma}
We have
\begin{displaymath}
H^*_{M(\C^d), \Gamma \times G \times T}(t; \alpha)=\frac{p(t; \alpha)}{\prod_{ij}
(1-\alpha_i \psi_j t)}.
\end{displaymath}
Here $p$ belongs to $K(\Gamma \times G \times T)[t]=K(\Gamma \times G)[t, \alpha_i]$.  The product is taken over
$1 \le i \le d$ and $1 \le j \le \dim{U}$.
\end{lemma}

\begin{proof}
As before, we can reduce to the case $M=A$ by considering the minimal resolution of $M$.  An easy computation, similar
to the previous one, gives $H^*_{A(\C^d), \Gamma \times G \times T}=\prod_{ij} (1-\alpha_i \psi_j t)^{-1}$.
\end{proof}

Lemma~\ref{hilbertlem} carries over exactly to the present situation.  Combining it with the previous lemma yields
\begin{equation}
\label{eq3}
H^*_{M, \Gamma \times G}(t)=\int_T \frac{p(t; \alpha)}{\prod_{ij} (1-\alpha_i \psi_j t)}
\frac{1}{1-(\sum \ol{\alpha}_i)} d\alpha.
\end{equation}
For $i$ fixed, we have
\begin{displaymath}
\frac{1}{\prod_j (1-\alpha_i \psi_j t)}=\sum \alpha_i^{\vert a \vert} \psi^a t^{\vert a \vert}
\end{displaymath}
where the sum is taken over $a \in \Z_{\ge 0}^n$, with $n=\dim{U}$, $\vert a \vert$ is defined as $a_1+\cdots+a_n$ and
$\psi^a$ is defined as $\psi_1^{a_1} \cdots \psi_n^{a_n}$.  Define
\begin{displaymath}
\qbinom{k}{n}{\psi}=\sum_{\vert a \vert=k} \psi^a,
\end{displaymath}
where $a$ belongs to $\Z_{\ge 0}^n$, so that
\begin{displaymath}
\frac{1}{\prod_j (1-\alpha_i \psi_j t)}=\sum_{k=0}^{\infty} \qbinom{k}{n}{\psi} \alpha_i^k t^k.
\end{displaymath}
With this notation in hand, we now expand the integrand of \eqref{eq3} into a series.  We find
\begin{displaymath}
H^*_{M, \Gamma \times G}(t)=\int_T \left[ \sum_{k, \ell} \qbinom{k}{n}{\psi} \alpha^k \left(
\sum \ol{\alpha_i} \right)^{\ell} p(t; \alpha) t^{\vert k \vert} \right] d\alpha.
\end{displaymath}
The sum is taken over $k \in \Z_{\ge 0}^d$ and $\ell \in \Z_{\ge 0}$ and our notation is as before; we mention
\begin{displaymath}
\qbinom{k}{n}{\psi}=\qbinom{k_1}{n}{\psi} \cdots \qbinom{k_d}{n}{\psi}.
\end{displaymath}
We must show that the previous equation is rational in $t$.  By linearity, it suffices to treat the case where
$p(t; \alpha)$ is of the form $x t^{e_0} \alpha_1^{e_1} \cdots \alpha_d^{e_d}$ where $x$ belongs to $K(\Gamma \times G)$
and the $e_i$ are integers.  Of course, $xt^{e_0}$ pulls out of the integral, and can thus be safely ignored.
Hence, it suffices to consider the case where $p$ is $\alpha^e$, with $e=(e_1, \ldots, e_d)$.  As before, the
$(k, \ell)$ term only contributes if $\ell=\vert k+e \vert$ and then only one term of $(\sum \ol{\alpha}_i)^{\ell}$
contributes.  The previous integral thus evaluates to:
\begin{displaymath}
\sum_k \qbinom{k}{n}{\psi} C_{k+e} t^{\vert k \vert}.
\end{displaymath}
Now, for a single integer $k$ the expression $\qbinom{k}{n}{\psi}$ is of the form $\sum_i a_i \psi_i^k$ where $a_i$
is a rational function of the $\psi$.  It follows that for $k \in \Z_{\ge 0}^d$ the expression $\qbinom{k}{n}{\psi}$
is of the form $\sum_i a_i \psi_i^k$ where the sum is taken over tuples $i \in \{1, \ldots, n\}^d$ and $\psi_i^k$
denotes $\psi_{i_1}^{k_1} \cdots \psi_{i_d}^{k_d}$; again the $a_i$ are rational functions in the $\psi$.  It thus
suffices to show that for each such tuple $i$, the expression
\begin{displaymath}
\sum_k \psi_i^k C_{k+e} t^{\vert k \vert}
\end{displaymath}
is rational in $t$.  This is accomplished in the following lemma, which completes the proof of
Theorem~\ref{hilbert2}:

\begin{lemma}
Keep the above notation and let $p$ be a polynomial.  Then
\begin{displaymath}
\sum p(k) \psi_i^k C_{k+e} t^{\vert k \vert}
\end{displaymath}
is a rational function of $t$.  (The sum is taken over $k \in \Z_{\ge 0}^d$.) Furthermore, the only denominators which
appear are of the form $1-at$ where $a$ is a sum of at most $d$ of the $\psi$'s.
\end{lemma}

\begin{proof}
As in the proof of Lemma~\ref{sumlem}, we reduce to the case where $p=1$.  We then change $k$ to $k-e$ and pull
monomials out of the sum, to obtain an expression of the form
\begin{displaymath}
\sum_{k \ge e} \psi_i^k C_k t^{\vert k \vert}.
\end{displaymath}
The difference
\begin{displaymath}
\sum_{k \ge 0} \psi_i^k C_k t^{\vert k \vert} - \sum_{k \ge e} \psi_i^k C_k t^{\vert k \vert}
\end{displaymath}
is a finite sum of sums of the form considered in the lemma, but with a smaller value of $d$.  (The terms which
appear may no longer have $p=1$; this is the reason for including $p$ in the general form of the sum we consider.)
It thus suffices to show that the first sum above is rational in $t$.  We have
\begin{displaymath}
\sum_{\vert k \vert=n} \psi_i^k C_k=(\psi_{i_1}+\cdots+\psi_{i_d})^n=a^n
\end{displaymath}
by the multinomial theorem.  Thus the first sum in the previous expression is $(1-at)^{-1}$, which completes the
proof.
\end{proof}

\subsection{Hilbert series in $\Sym(\mc{S})$}

Let $M$ be an object of $\Sym(\mc{S})$.  We can regard $M$ as a sequence of equivariant functors $M_n:\Vec^n \to \Vec$.
Write
\begin{displaymath}
M_n(V_1, \ldots, V_n)=\bigoplus_{i \in I_n} \bS_{\lambda_{i, 1}}(V_1) \otimes \cdots \otimes \bS_{\lambda_{i, n}}(V_n)
\end{displaymath}
for some index set $I_n$ and partitions $\lambda_{i, j}$ (both depending on $n$).  We assume $I_n$ is finite
for each $n$.  Put
\begin{displaymath}
m_n^*=\sum_{i \in I_n} s_{\lambda_{i, 1}} \cdots s_{\lambda_{i, n}}, \qquad H^*_M(t)=\sum_{n=0}^{\infty} m_n^* t^n.
\end{displaymath}
We regard $m_n^*$ as an element of the polynomial ring $\Q[s_{\lambda}]$ and $H_M^*(t)$ as an element of the
power series ring $\Q[s_{\lambda}] \lbb t \rbb$.  The variable $t$ is basically superfluous, since the power of
$t$ can be obtained from the order of the polynomial $m_n^*$.  When $t$ is omitted (or set to 1), $H_M^*$
agrees with $[M]^*$.  One can also define $H_M(t)$, which is analogous to $[M]$, by replacing $m_n^*$ with
$m_n=\tfrac{1}{n!} m_n^*$.  Our main result concerning these series is the following:

\begin{theorem}
\label{hilbert3}
Let $A$ be an algebra in $\Sym(\mc{S})$ finitely generated in order 1 on which a finite group $\Gamma$ acts and let
$M$ be a finitely generated $A$-module with a compatible action of $\Gamma$.  Then $H^*_{M^{\Gamma}}$ is a rational
function of the $s_{\lambda}$.
\end{theorem}

As with Theorem~\ref{hilbert2}, this result does not translate to an elegant statement about $H_{M^{\Gamma}}$.  We now
explain the basic strategy of our proof, in the case where $\Gamma$ is trivial.  Let $M$ be a finitely generated
$A$-module.  We would like to relate $H^*_M$ to the Hilbert series of an object of $\Sym(\Vec)$, so that we can
apply the results we have established in that case.  The most obvious way to obtain an object of $\Sym(\Vec)$ is to
pick a vector space $U$, let $i:\Cfs \to \Vec^f$ be the functor assigning to $L$ the constant family $U_L$ and then
consider $i^*M$.  Unfortunately, $H^*_M$ cannot be recovered from $H^*_{i^*M}$, or even $H^*_{i^*M, \GL(U)}$.  Thus the
most obvious approach fails.  However, a slight modification works:  instead of picking just one vector space $U$ we
pick finitely many $U_1, \ldots, U_r$ and build from $M$ an object of $\Sym(\Vec)$ with an action of $\GL(U_1) \times
\cdots \times \GL(U_r)$.  It turns out that, due to the form of $M$, we can take $r$ large enough so that information
is not lost --- this is essentially the content of Proposition~\ref{hilbert3-4} below.  Before proving that result,
we need some lemmas, the first two of which are left to the reader.

\begin{lemma}
\label{hilbert3-1}
Let $f:\C^n \to \C^m$ be a polynomial map whose components are homogeneous of positive degree and whose image is not
contained in any linear subspace of $\C^m$.  Then for $r \gg 0$ any element of $\C^m$ can be expressed as a sum of $r$
elements of the image of $f$.
\end{lemma}

\begin{lemma}
\label{hilbert3-2}
Let $K$ be a field and let $(f_i)$ be a sequence of elements in $K$.  Assume that there exists $m>0$ such that
for all $k_1, \ldots, k_m$ sufficiently large, the $m \times m$ matrix $(f_{k_i-j})$ has determinant zero.  Then
$\sum_{i \ge 0} f_i t^i$ can be expressed in the form $a/b$ where $a$ and $b$ belong to $K[t]$ and $\deg{b} \le m-1$.
\end{lemma}

\begin{lemma}
\label{hilbert3-3}
Let $A$ be a UFD and let $(f_i)$ be a sequence of elements of $A$.  Assume that there exists an integer $m$ and
elements $\alpha_1, \ldots, \alpha_m$ of $A$ such that
\begin{displaymath}
f_k=\sum_{i=1}^m \alpha_i f_{k-i}
\end{displaymath}
holds for all $k \gg 0$.  Amongst all such expressions choose one with $m$ minimal.  Then given any $N$ there
exist $k_1, \ldots, k_m > N$ such that the determinant of the matrix $(f_{k_i-j})$ is non-zero.
\end{lemma}

\begin{proof}
Put $f=\sum f_i t^i$, an element of $A \lbb t \rbb$, and $q=1-\sum_{i=1}^m \alpha_i t^i$, an element of $A[t]$.
We have that $qf$ belongs to $A[t]$; write $p=qf$ so that $f=p/q$.  Note that $A[t]$ is also a UFD and that the
minimality assumption on $m$ implies that $p$ and $q$ are coprime.  Indeed, say $p$ and $q$ are both divisible
by some non-unit $r \in A[t]$.  Then we can write $p=rp'$ and $q=rq'$.  Evaluating the second expression at $t=0$
gives $1=r(0) q'(0)$ so that $r(0)$ and $q'(0)$ both belong to $A^{\times}$.  We can therefore scale $r$ so that
$r(0)=1$, in which case $q'(0)=1$ as well.  Since $r$ is assumed to be a non-unit and $r(0)$ is a unit, it follows
that $r$ has degree at least 1.  We then have $f=p'/q'$ with $q'(0)=1$ and $\deg{q}'<m$, contradicting the minimality
of $m$.

Now, assume for the sake of contradiction that $\det(f_{k_i-j})=0$ for all sufficiently large $k_i$.  By
Lemma~\ref{hilbert3-2} we have $f=g/h$ where $g$ and $h$ belong to $K[t]$ and $\deg{h}<m$.  Here $K$ is the field of
fractions of $A$.  Pick $a \in A$ non-zero so that $ag$ and $ah$ belong to $A[t]$.  We then have $(ah)p=(ag) q$.  Since
$p$ is coprime to $q$ it follows that $ah$ is divisible by $q$.  However, this contradicts $h$ having smaller degree
than $q$.  We thus conclude that $\det(f_{k_i-j})$ cannot vanish for all sufficiently large $k_i$.
\end{proof}

\begin{proposition}
\label{hilbert3-4}
Let $W$ be a finite dimensional vector space over a subfield $\C$ and let $V$ be a finite dimensional
subspace of $\Sym(W)$ spanned
by homogeneous elements.  For a positive integer $r$, let $i_r:\Sym(V) \to \Sym(W)^{\otimes r}$ be the ring
homomorphism which on $V$ is given by $x \mapsto \sum \{x\}_i$, where $\{-\}_i:\Sym(W) \to \Sym(W)^{\otimes r}$ is the
ring homomorphism given by inclusion into the $i$th factor.  Then for $r \gg 0$ we have:
\begin{enumerate}
\item The map $i_r$ is injective.
\item If $x \in \Frac(\Sym(V))$ and $i_r(x)$ belongs to $\Sym(W)^{\otimes r}$ then $x$ belongs to $\Sym(V)$.
\item A series $f \in \Sym(V) \lbb t \rbb$ is rational if and only if $i_r(f)$ is.
\end{enumerate}
\end{proposition}

\begin{proof}
We can check the conclusions of the proposition by tensoring up to $\C$, and so we may thus assume we are working with
complex vector spaces.  The map $\Sym(V) \to \Sym(W)$ corresponds to an algebraic map $f:W^* \to V^*$ whose components
are homogeneous of positive degree.  Since $V \to \Sym(W)$ is injective, the image of $f$ is
not contained in any linear subspace of $V^*$.  It thus follows from Lemma~\ref{hilbert3-1} that for $r \gg 0$, any
element of $V^*$ is a sum of $r$ elements of the image of $f$.  Now, the map $i_r$ corresponds to the map $i_r^*:
(W^*)^r \to V^*$ given by $(w_1, \ldots, w_n) \mapsto \sum f(w_i)$.  We thus see that $i_r^*$ is surjective for all
$r \gg 0$.

Fix $r \gg 0$ and put $i=i_r$.  We now show that (a) and (b) hold.  The equation $i(x)=0$ is
equivalent to $x \circ i^*=0$, where $x$ is thought of as a function $V^* \to \C$.  Since $i^*$ is surjective, this
equation implies $x=0$.  This shows that $i$ is injective.  Now say that $x$ belongs to $\Frac(\Sym(V))$ and $i(x)$
is a polynomial.  Then $x$ defines a rational function on $V^*$ such that $x \circ i^*$ is a regular function on
$(W^*)^r$.  Since $i^*$ is surjective, $x$ must be regular on $V^*$ and thus it belongs to $\Sym(V)$.

We now prove (c).  It is clear that if $f$ is rational then $i(f)$ is as well.  Thus let $f \in \Sym(V) \lbb t \rbb$
be given and assume that $i(f)$ is rational.  Write $f=\sum f_i t^i$ with $f_i \in \Sym(V)$.  The rationality of
$i(f)$ means that we can find a polynomial $q=1-\sum_{i=1}^m \alpha_i t^i$ with $\alpha_i \in \Sym(W)^{\otimes r}$
such that $qf$ is a polynomial.  Choose $q$ with $m$ minimal.  We then have
\begin{displaymath}
f_k = \sum_{i=1}^m \alpha_i f_{k-i}
\end{displaymath}
for all sufficiently large $k$.  Thus for all large $k_1, \ldots, k_m$ we have the equation $Ax=y$ where
$A$ is the $m \times m$ matrix $(f_{k_i-j})$, $x$ is the column vector $(\alpha_i)$ and $y$ is the column vector
$(f_{k_i})$.  By the minimality of $m$ and Lemma~\ref{hilbert3-3} we can pick $k_i$ so that $\det{A}$ is non-zero.
We then find $x=A^{-1} y$, which shows that $\alpha_i$ belongs to $i(\Frac(\Sym(V)))$.  Since $\alpha_i$ also belongs
to $\Sym(W)^{\otimes r}$, statement (b) of the lemma implies that each
$\alpha_i$ belongs to $i(\Sym(V))$.  Thus $q=i(q')$ for a unique $q'$ in $\Sym(V)[t]$ with the required properties
to establish that $f$ is rational.
\end{proof}

We now prove the theorem:

\begin{proof}[Proof of Theorem~\ref{hilbert3}]
Let $\Gamma$, $A$ and $M$ be given.  Let $P$ be the set of partitions appearing in $M$.  The set $P$ is finite; indeed,
only finitely many partitions appear in $A$ (see the discussion preceding Theorem~\ref{wnnoeth}) and $M$ is a
quotient of $A \otimes F$ for some finite length object $F$ of $\Sym(\mc{S})$.  Let $V$ be the subspace of $K(\mc{S})$
spanned by the $s_{\lambda}$ with $\lambda \in P$.  Thus $H^*_{M^{\Gamma}}$ belongs to $\Sym(V) \lbb t \rbb$.

Let $U$ be a finite dimensional vector space whose dimension exceeds the number of rows of any partition appearing
in $P$ and let $G=\SL(U)$.  Then $K(G)$ is a polynomial ring.  (We use $\SL(U)$ instead of $\GL(U)$ so that the
Grothendieck group is a polynomial ring; it makes the argument a bit cleaner.)  Evaluation on $U$ gives a map
$K(\mc{S}) \to K(G)$ which is injective when restricted to $V$.  Let $r$ be a large integer and let
\begin{displaymath}
\phi : \Sym(K(\mc{S})) \to K(G)^{\otimes r}
\end{displaymath}
be the ring map which is given by
\begin{displaymath}
\phi([S])=\sum_{i=1}^r \{[S(U)]\}_i
\end{displaymath}
for $[S]$ in $K(\mc{S})$, where $\{-\}_i:K(G) \to K(G)^{\otimes r}$ is the inclusion in the $i$th factor.  By
Proposition~\ref{hilbert3-4}(c), a power series $f \in \Sym(V) \lbb t \rbb$ is a rational function if and only if $\phi(f)$
is.  It thus suffices to show that $\phi(H^*_{M^{\Gamma}})$ is rational.

To understand the map $\phi$ we lift it to a functor $\Phi$, as follows.  First,
identify $K(G)^{\otimes r}$ with $K(G^r)$.  Let $U_1, \ldots, U_r$ be copies of $U$ and put $G_i=\SL(U_i)$; we think of
$G^r$ as $G_1 \times \cdots \times G_r$.  We regard $\phi$ as a map
\begin{displaymath}
\phi : \Sym(K(\mc{S})) \to K(G_1 \times \cdots \times G_r).
\end{displaymath}
It can be described explicitly on $K(\mc{S})$ as follows:
\begin{displaymath}
\phi([S])=\sum_{i=1}^r [S(U_i)].
\end{displaymath}
Now define $\Phi$ to be the functor
\begin{displaymath}
\Phi:\Sym(\mc{S}) \to \Sym(\Vec), \qquad \Phi(F)(L)=\bigoplus_{L=L_1 \amalg \cdots \amalg L_r} F((U_1)_{L_1} \amalg
\cdots \amalg (U_r)_{L_r}).
\end{displaymath}
Here the sum is over all partitions of $L$ into $r$ parts and $(U_i)_{L_i}$ denotes the family $(V, L_i)$ where
$V_x=U_i$ for all $x \in L_i$.  (A more conceptual description of $\Phi$ is as follows.  Let $i:\Cfs^r \to \Vec^f$
take $(L_1, \ldots, L_r)$ to $(U_1)_{L_1} \amalg \cdots \amalg (U_r)_{L_r}$ and let $j:\Cfs^r \to \Cfs$ be the addition
map.  Then $\Phi(F)=j_*i^*F$.)  One readily verifies that $\Phi$ is a tensor functor.  We now claim that $\Phi$ lifts
$\phi$, that is, we have
\begin{displaymath}
\phi([N])=[\Phi(N)]
\end{displaymath}
for all $N$ in $\Sym(\mc{S})$.  Both sides above are additive in $N$ so it suffices to treat the case
where $N$ is a simple object of $\Sym(\mc{S})$.  As we have previously stated (\S \ref{ss:symcat}), the simple objects
are of the form
$\bigotimes \bS_{\lambda_i}(S_i)$ where the $\lambda_i$ are partitions and $S_i$ are distinct simple objects of
$\mc{S}$.  Since $\Phi$ is a tensor functor and $\phi$ is a ring homomorphism, we are reduced to the case of
considering $N=\bS_{\lambda}(S)$.  Put $k=\vert \lambda \vert$.  Then $N$ is the equivariant functor
$\Vec^k \to \Vec$ given by
\begin{displaymath}
(V_1, \ldots, V_k) \mapsto \Sp_{\lambda} \otimes S(V_1) \otimes \cdots \otimes S(V_k).
\end{displaymath}
We thus find that $\Phi(N) \in \Sym(\Vec)$ is supported in order $k$ and assigns to the set $\{1, \ldots, k\}$
the space
\begin{displaymath}
\Sp_{\lambda} \otimes \bigoplus_{i_1, \ldots, i_k} S(U_{i_1}) \otimes \cdots \otimes S(U_{i_k})
\end{displaymath}
where the sum is over all $(i_1, \ldots, i_k)$ in $\{1, \ldots, r\}^k$.  We therefore have
\begin{displaymath}
[\Phi(N)]=\frac{\dim{\Sp_{\lambda}}}{k!} \left( \sum_{i=1}^r [S(U_i)] \right)^k.
\end{displaymath}
On the other hand,
\begin{displaymath}
[N]=\frac{\dim{\Sp_{\lambda}}}{k!} [S]^k.
\end{displaymath}
Applying $\phi$ to the above gives exactly the previous formula for $[\Phi(N)]$.  This proves the claim.

The above discussion, and the fact that $\Phi$ commutes with the formation of $\Gamma$ invariants, shows that
\begin{displaymath}
\phi(H_{M^{\Gamma}}^*)=H_{\Phi(M)^{\Gamma}, G_1 \times \cdots \times G_r}^*.
\end{displaymath}
As $\Phi$ is a tensor functor, $\Phi(A)$ is a twisted commutative algebra finitely generated in order 1 and
$\Phi(M)$ is a finitely generated module over it.  Thus $H^*_{\Phi(M)^{\Gamma}, G_1 \times \cdots \times G_r}$
is a rational function in $K(G_1 \times \cdots \times G_r) \lbb t \rbb$ by Theorem~\ref{hilbert2}.  We thus find
that $\phi(H_{M^{\Gamma}}^*)$ is rational, which completes the proof.
\end{proof}

We note the following corollary of the proposition.

\begin{corollary}
\label{hilbert4}
Let $A$ be an algebra in $\Sym(\mc{S})$ finitely generated in order 1 on which a finite group $\Gamma$ acts and let
$M$ be a finitely generated $A^{\Gamma}$-module.  Then $H_M^*$ is a rational function of the $s_{\lambda}$.
\end{corollary}

\begin{proof}
Let $N=A \otimes_{A^{\Gamma}} M$.  Then $N$ is a finitely generated $A$-module and $N^{\Gamma}=M$, where $\Gamma$ acts
on $N$ by acting trivially on $M$.  We thus have that $H^*_M=H^*_{N^{\Gamma}}$ is rational by the proposition.
\end{proof}

\subsection{Hilbert series of $\Delta$-modules}

Let $M$ be a $\Delta$-module.  We define its Hilbert series to be the Hilbert series of the underlying object of
$\Sym(\mc{S})$.  Our main result on such Hilbert series is the following, which follows immediately from
Corollary~\ref{hilbert4} and Proposition~\ref{delta-filt}.

\begin{theorem}
\label{hilbert5}
Let $M$ be a small $\Delta$-module.  Then $H^*_M$ is a rational function of the $s_{\lambda}$.
\end{theorem}

\begin{remark}
We expect this result to hold for all finitely generated $\Delta$-modules.  We can prove it for a much larger class
than the class of small $\Delta$-modules, but we have not been able to prove it for all finitely generated
$\Delta$-modules.
\end{remark}

\section{Syzygies of $\Delta$-schemes}

We now apply the theory we have developed to the study of the syzygies of certain families of schemes, which we
call $\Delta$-schemes.

\subsection{$\Delta$-schemes}
\label{ss:delta}

An \emph{abstract $\Delta$-scheme} is a functor $X$ from the category $\Vec^{\Delta}$ to the category of schemes
over $\C$.  The notion of a morphism of abstract $\Delta$-schemes is evident.  In this way we have a category of
abstract $\Delta$-schemes.  We say that a morphism $X \to Y$ of abstract $\Delta$-schemes is a closed immersion if
$X(V, L) \to Y(V, L)$ is a closed immersion for all $(V, L)$.

The category of abstract $\Delta$-schemes is too large for our purposes; we now introduce a more manageable category.
For an object $(V, L)$ of $\Vec^{\Delta}$, let $\bV(V, L)$ be the vector space $\bigotimes_{x \in L} V_x^*$, regarded
as a scheme.  Then $\bV$ is an abstract $\Delta$-scheme, in an obvious way.  A \emph{$\Delta$-scheme} is a pair
$(X, i)$ consisting of an abstract $\Delta$-scheme $X$ such that $X(V, L)$ is non-empty for all $(V, L)$ and a closed
immersion $i:X \to \bV$.  Note that if $X$ is a $\Delta$-scheme, then $X(V, L) \subset \bV(V, L)$ is closed under
scaling (by functoriality), and is thus a cone.  A morphism of $\Delta$-schemes is a morphism of abstract
$\Delta$-schemes which commutes with the embeddings into $\bV$.
There is at most one morphism between two $\Delta$-schemes, and so the category of $\Delta$-schemes is
partially ordered.

We can think of $\Delta$-schemes in terms of their ideals of definitions, as follows.  Let
$P(V, L)$ be the affine coordinate ring of $\bV(V, L)$.  We have
\begin{displaymath}
P(V, L) = \Sym \bigg( \bigotimes_{x \in L} V_x \bigg).
\end{displaymath}
A \emph{$\Delta$-ideal} is a rule which assigns to each object $(V, L)$ of $\Vec^{\Delta}$ an ideal $I(V, L)$ of
$P(V, L)$ such that if $(V, L) \to (V', L')$ is a morphism in $\Vec^{\Delta}$ then under the map $P(V, L) \to
P(V', L')$ the ideal $I(V, L)$ maps into $I(V', L')$.  Suppose $I$ is a $\Delta$-ideal.  Let $X(V, L)$ be the
subscheme of $\bV(V, L)$ cut out by $I(V, L)$.  Then $X$ is naturally a $\Delta$-scheme.  The association
$I \mapsto X$ is an anti-equivalence of categories.  By an \emph{element} of $P$ we mean an element of $P(V, L)$ for
some $(V, L)$.  Given any collection $S$ of elements of $P$ there is a unique minimal $\Delta$-ideal $I$ containing
$S$; we call $I$ the $\Delta$-ideal \emph{generated} by $S$.

In more concrete terms, a $\Delta$-scheme can be thought of as a sequence of rules $(X_n)_{n \ge 0}$, where $X_n$
assigns to each $n$-tuple of vector spaces $(V_1, \ldots, V_n)$ a closed subscheme $X_n(V_1, \ldots, V_n)$ of
$\bigotimes V_i^*$, such that:  $X_n$ is functorial in $(V_1, \ldots, V_n)$; $X_n$ is $S_n$-equivariant; and
$X_{n+1}(V_1, \ldots, V_{n+1})$ is contained in $X_n(V_1, \ldots, V_{n-1}, V_n \otimes V_{n+1})$.  Of course, we can
describe $\Delta$-ideals similarly.

\subsection{Examples of $\Delta$-schemes}

We now give some examples of $\Delta$-schemes.  There are two rather trivial examples:  $\bV$ itself is a
$\Delta$-scheme, and the final object in the category; and the rule $(V, L) \mapsto \Spec(\C)$ (embedded in $\bV$ as
the origin) is a $\Delta$-scheme, and the initial object in the category.  A more interesting example, and the one
which motivated the whole theory, is the Segre embedding:  take $X(V, L)$ to be the set of pure tensors in
$\bV(V, L)=\bigotimes_{x \in L} V_x^*$, with its usual scheme structure.

In fact, the Segre example is just the first in a family.  Let $d \ge 0$ be an integer.  As a vector space, the algebra
$P(V, L)$ breaks up into isotypic pieces for the action of the group $\prod_{x \in L} \GL(V_x)$; each such
piece corresponds to a family of partitions indexed by $L$.  Let $I_d(V, L)$ be the sum of all
pieces in which some partition has more than $d$ rows.  One easily checks that $I_d(V, L)$ is an ideal of $P(V, L)$,
and so cuts out a closed subscheme $\Sub_d(V, L)$ of $\bV(V, L)$.  Geometrically, a point $v$ in $\bV(V, L)$ belongs to
$\Sub_d(V, L)$ if and only if for each $x \in L$ there exists a quotient $U_x$ of $V_x$ of dimension at most $d$
such that $v$ belongs to $\bigotimes_{x \in L} U_x^*$.  For further discussion of these subspace varieties, see
\cite[\S 3]{LW2}.

Now, it is clear that $\Sub_d(V, L)$ is functorial for maps in $\Vec^f$.  However, for $d \ge 2$, it is not functorial
for maps in $\Vec^{\Delta}$, and is therefore not a $\Delta$-scheme.  Nonetheless, $\Sub_d$ contains a unique maximal
$\Delta$-scheme, which we call $\dSub_d$.  Geometrically, we have
\begin{displaymath}
\dSub_d(V, L)=\bigcap \Sub_d(V, \ms{U}),
\end{displaymath}
where the scheme-theoretic intersection is taken over all partitions $\ms{U}$ of $L$.
Algebraically, $\dSub_d$ corresponds to the $\Delta$-ideal generated by the $I_d(V, L)$.  It follows from
\cite[Thm~3.1]{LW2} that the $\Delta$-ideal corresponding to $\dSub_{d-1}$ (for $d>1$) is generated by any element of
the one dimensional space
\begin{displaymath}
\lw{d}{\C^d} \otimes \lw{d}{\C^d} \subset \Sym^d(\C^d \otimes \C^d) \subset P_2(\C^d, \C^d).
\end{displaymath}
One can describe $\dSub_d$ as the largest $\Delta$-scheme whose coordinate ring has the property that every partition
appearing in it has at most $d$ rows.  

The $\Delta$-scheme $\dSub_0$ is just the initial $\Delta$-scheme, i.e., the origin in $\bV$.
The $\Delta$-scheme $\dSub_1$ coincides with $\Sub_1$, and is the Segre embedding.  The $\Delta$-scheme
$\dSub_2$ (which does not coincide with $\Sub_2$) is equal to the secant variety of the Segre embedding --- this
assertion is essentially equivalent to the recently proved GSS conjecture (see \cite{Raicu}).  We do not know an elegant
geometric description of $\dSub_d$ for $d>2$.

The category of $\Delta$-schemes is closed under several natural operations.  Let $X$ and $Y$ be two $\Delta$-schemes.
Then the union $X \cup Y$ and scheme-theoretic intersection $X \cap Y$ of $X$ and $Y$ inside of $\bV$ are
$\Delta$-schemes.  In this way, the poset of $\Delta$-schemes is a lattice.  Let $X+Y$ be the scheme-theoretic
image of the map $X \times Y \to \bV$ given by $(x, y) \mapsto x+y$.  Then $X+Y$ is a $\Delta$-scheme.  In particular,
the secant schemes to a $\Delta$-scheme are again $\Delta$-schemes.  The same holds for tangent
schemes.  The functor $X_{\red}$ which attaches to $(V, L)$ the reduced subscheme $X_{\red}(V, L)$ of
$X(V, L)$ is a $\Delta$-scheme.
Starting with the $\dSub_d$ and using these operations, one obtains a large number of examples of
$\Delta$-schemes.

\subsection{Syzygies of $\Delta$-schemes}
\label{ss:delta-syz}

Fix a $\Delta$-scheme $X$.  Let $R(V, L)$ be the affine coordinate ring of $X(V, L)$, a quotient of the polynomial
ring $P(V, L)$.  Put
\begin{displaymath}
F_p(V, L)=\Tor_p^{P(V, L)}(R(V, L), \C).
\end{displaymath}
As discussed in the introduction, the $F_p(V, L)$ record the syzygies of $R(V, L)$ as a $P(V, L)$-module.  The
functorial properties of $R$, $P$ and $\Tor$ imply that $F_p$ is a $\Delta$-module; this will be made more clear in
the proof of the following theorem.  Let $F_p^{(d)}$ be its
degree $d$ piece; it too is a $\Delta$-module.  Our main result on syzygies is the following theorem:

\begin{theorem}
\label{thm:delta-small}
The $\Delta$-module $F_p^{(d)}$ is small.
\end{theorem}

\begin{proof}
Let $W(V, L)=\bigotimes_{x \in L} V_x$ be the degree one piece of $P(V, L)$.  The $P(V, L)$-module $\C$ admits a
Koszul resolution, the terms of which are $P(V, L) \otimes \bw{i}{W(V, L)}$.  Tensoring this over $P(V, L)$ with
$R(V, L)$, we find that there is a complex computing $F_p(V, L)$ whose terms are $R(V, L) \otimes \bw{i}{W(V, L)}$.
Put
\begin{displaymath}
M_{i,j}(V, L)=R^{(i)}(V, L) \otimes \bw{j}{W(V, L)}.
\end{displaymath}
We then find that $F_p^{(d)}(V, L)$ is the homology of a complex
\begin{displaymath}
M_{d-p-1,p+1}(V, L) \to M_{d-p, p}(V, L) \to M_{d+1-p, p-1}(V, L)
\end{displaymath}
Now, the functor $(V, L) \mapsto M_{i,j}(V, L)$ is easily seen to be a $\Delta$-functor.  A short calculation with
the Koszul complex shows that the differentials in the above complex respect the $\Delta$-module structure on the
$M_{i,j}$.  Thus $F_p^{(d)}$, as a $\Delta$-module, is the homology of the complex of $\Delta$-modules
\begin{displaymath}
M_{d-p-1,p+1} \to M_{d-p,p} \to M_{d+1-p,p-1}.
\end{displaymath}
We now claim that $M_{i,j}$ is a small $\Delta$-functor.  Indeed, $R^{(i)}$ is a quotient of $P^{(i)}=\Phi(\Sym^i)$,
which is a quotient of $\Phi(T_i)$.  Of course, $(V, L) \mapsto \bw{j}{W(V, L)}$ is the $\Delta$-module
$\Phi(\lw{j}{})$, and thus a quotient of $\Phi(T_j)$.  Thus $M_{i,j}=R^{(i)} \boxtimes \Phi(\lw{j}{})$ a quotient of
$\Phi(T_i) \boxtimes \Phi(T_j)=\Phi(T_{i+j})$, and is thus small.  Now, a subquotient of a small
$\Delta$-module is small; thus $F_p^{(d)}$ is small.  This completes the proof.
\end{proof}

\begin{corollary}
\label{cor:delta-fg}
The $\Delta$-module $F_p^{(d)}$ is finitely generated.
\end{corollary}

Thus there exists a finite list of $p$-syzygies of degree $d$ for the schemes $X(V, L)$ which give rise to all
$p$-syzygies of degree $d$ under the basic operations on the $\Delta$-module $F_p^{(d)}$.  In view of this corollary,
the statement ``$F_p$ is finitely generated as a $\Delta$-module'' is equivalent to the statement ``$F_p^{(d)}=0$
for $d \gg 0$.''  This probably does not hold in full generality, but in \S \ref{s:finlev}  we single out a large
class of $\Delta$-schemes for which it might hold.

\begin{corollary}
The series $[F_p^{(d)}]^*$ is a rational function in the $s_{\lambda}$.
\end{corollary}

This essentially shows that the information content of the $p$-syzygies of degree $d$ of $X$ is finite.  As we have
said, this series can be computed algorithmically (this is discussed below).  Thus, essentially all the information
about $p$-syzygies of degree $d$ of $X$ can be computed in finite time!

\subsection{Computational aspects}
\label{ss:comp}

We now elaborate on the remark we have made that our proofs give algorithms to calculate the relevant objects.  Keep
the same notation as the previous section.  Say we would like to compute generators for the $\Delta$-module $F_p^{(d)}$
of $p$-syzygies of degree $d$.  (The algorithm for computing the rational function $[F_p^{(d)}]^*$ proceeds along
similar lines.)  We proceed as follows.  First, $F_p^{(d)}$ is the homology of the sequence
\begin{displaymath}
M_{d-p-1, p+1} \to M_{d-p, p} \to M_{d+1-p, p-1}.
\end{displaymath}
As explained, $R^{(i)}$ is a quotient of $\Phi(T_i)$ while $(V, L) \mapsto \bw{j}{W(V, L)}$ is a quotient of
$\Phi(T_j)$, and so $M_{i,j}$ is a quotient of $\Phi(T_{i+j})$.  In particular, each module in the above complex is
a quotient of $\Phi(T_d)$.  This shows that each is canonically a $W_d^{S_d}$-module and the differentials respect this
structure.  Now, $W_d$ only has partitions with at most $d$ rows; the same is true for the $M_{i,j}$ above.  Thus it
suffices to see what happens when we evaluate on $\C^d$.  Precisely, let $A$ be the twisted commutative algebra
$L \mapsto W_d^{S_d}((\C^d)_L)$ and let $N_{i,j}$ be the $A$-module given by $L \mapsto M_{i,j}((\C^d)_L)$, where
$(\C^d)_L$ denotes the constant family on $\C^d$.  Let $E$ be the homology of
the complex
\begin{displaymath}
N_{d-p-1, p+1} \to N_{d-p, p} \to N_{d+1-p, p-1};
\end{displaymath}
thus $E_L=F_p^{(d)}((\C^k)_L)$.  The row bounds then imply that generators for $E$
as an $A$-module are generators for $F_p^{(d)}$ as a $W_d^{S_d}$-module.  Proposition~\ref{delta-w} thus implies that
generators for $E$ as an $A$-module are generators for $F_p^{(d)}$ as a $\Delta$-module.

Now, $A$ is equal to $\Sym(U\langle 1 \rangle)^{S_d}$ where $U=(\C^d)^{\otimes d}$.  It follows that any
partition in $A$ has at most $\dim{U}=d^d$ rows.  Since $M_{d-i,i}$ is a subquotient of $W_d$, it follows that
$N_{d-i,i}$ is a subquotient of $\Sym(U \langle 1 \rangle)$.  Thus any partition appearing in $N_{d-i, i}$ has at most
$d^d$ rows as well.  Thus we do not loose information by evaluating on $\C^{d^d}$ (regarding everything in the Schur
model).  That is, generators for $E(\C^{d^d})$ as an $A(\C^{d^d})$-module give generators for $E$.

Finally, $A(\C^{d^d})$ is the subring of the polynomial ring in $d^{2d}$ variables which are $S_d$-invariant.  Each
of the modules $N_{d-i,i}(\C^{d^d})$ is a finite module over this ring.  And $E(\C^{d^d})$ is the homology of the
complex $N_{\bullet}$ at $i=p-1$.  We have thus reduced the problem to a computation involving explicitly described
finitely generated rings and modules.  These computations can be done algorithmically.

\subsection{$\Delta$-schemes of finite level}
\label{s:finlev}

A $\Delta$-scheme $X$ (resp.\ $\Delta$-ideal $I$) has \emph{level $\le d$} if every partition appearing in its
coordinate ring has at most $d$ rows.  One easily sees that $X$ has level $\le d$ if and only if it is contained in
$\dSub_d$.  We say that $X$ has \emph{finite level} if it has level $\le d$ for some $d$.
All the operations we have discussed (union, intersection, sum, formation of secant, tangent and reduced schemes)
preserve the finite level condition.  We thus have many examples of $\Delta$-schemes of finite level.  Note, however,
that the final $\Delta$-scheme $\bV$ does not have finite level.

The following three statements are equivalent:
\begin{enumerate}
\item Every ascending chain of $\Delta$-ideals of finite level stabilizes.
\item Every $\Delta$-ideal of finite level is finitely generated.
\item Every $\Delta$-ideal $I$ of finite level is ``uniformly generated,'' i.e., there exists an integer $d$ such
that $I(V, L)$ is generated in degrees $\le d$ for all $(V, L)$.
\end{enumerate}
It is clear that (a) and (b) are equivalent, and clear that each implies (c).  That (c) implies (b) follows
from Corollary~\ref{cor:delta-fg} and the discussion following it.

It seems reasonable to hope that statements (a)--(c) are in fact true.  Indeed, there has been some recent work
establishing that particular finite level $\Delta$-ideals are finitely generated:
\begin{itemize}
\item Raicu \cite{Raicu} has proved the GSS conjecture, which implies that the ideal of the first secant
variety to the Segre embedding is finitely generated as a $\Delta$-ideal.
\item Draisma and Kuttler \cite{DraismaKuttler} have shown that the secant varieties to the Segre varieties are
all cut out set-theoretically by equations of a bounded degree.  (In fact, in \cite[\S 7]{DraismaKuttler}, a conjecture
concerning something similar to statements (a)--(c) above is posed.)
\item Landsberg and Weyman \cite{LW} have obtained results about the ideal of the tangent variety to the Segre, which
imply that it is finitely generated as a $\Delta$-ideal.
\end{itemize}
A proof of the statements (a)--(c) would have great consequences:  for instance, it would essentially encompass all of
the above results.  Assuming these statements are true, it would be reasonable to believe that the $p$th syzygy module
of a finite level $\Delta$-scheme is supported in finitely many degrees.

\section{Application to Segre varieties}

We now apply the theory we have developed to the study of syzygies of Segre embeddings.

\subsection{Theorems~A and~B}

We begin by establishing the theorems stated in the introduction.  Let $F_p$ be the $\Delta$-module of $p$-syzygies
of the Segre embedding, as defined in \S \ref{ss:delta-syz}.  By Theorem~\ref{thm:delta-small} and its corollaries,
Theorems~A and~B follow from the following result, which states that $F_p$ has only finitely
many non-zero graded pieces.  This result is well-known to the experts, so we only give a brief proof.  We thank
Aldo Conca for showing it to us.

\begin{proposition}
\label{degbd}
For $p \ge 1$, the space $F_p$ is supported in degrees $p+1, \ldots, 2p$.
\end{proposition}

\begin{proof}
This follows from the fact that the ideal of the Segre variety is generated by a Gr\"obner basis of degree 2
\cite[Prop.~17]{ERT}, together with general facts about Gr\"obner bases (such as the Taylor resolution, see
\cite[Exercise~17.11]{Eisenbud}).
\end{proof}

\begin{remark}
The bound in the proposition is not optimal.  Indeed, it is known \cite{Rubei} that $F_2$ and $F_3$ are supported in
degrees 3 and 4, while the upper bounds provided by the proposition are 4 and 6.  The optimal upper bound is not known.
However, our proof of Theorem~\ref{thm2} provides an algorithm for finding it for any particular value of $p$.
\end{remark}

\subsection{The work of Lascoux}
\label{ss:lascoux}

Lascoux determined the entire minimal resolution of certain determinantal varieties (see \cite{Lascoux}, and also
\cite{PW}, where a gap in \cite{Lascoux} is resolved).  The rank 1 case of his result exactly gives the
leading term of our series $f_p$.  We recall his results in our language.

First, we discuss some terminology.  Let $m=s_{\lambda_1} \cdots s_{\lambda_n}$ be a monomial in the variables
$s_{\lambda}$.  We say that $m$ has \emph{order} $n$.  We say that $m$ has \emph{degree} $d$ if each $\lambda_i$ is
a partition of $d$.  Every term in the series $[F_p^{(d)}]$ has degree $d$, while the orders of the terms are
unbounded.  The degree $d$, order $n$ terms in $f_p$ give information about the $p$-syzygies of degree $d$ for the
Segre embedding of an $n$-fold product of projective spaces.

Let $f_{p, 2}$ be the order two two term of $f_p$.  This is the leading order term of $f_p$.  We consider its degree
$d$ piece $f^{(d)}_{p, 2}$.  Write $d=p+h$.  Of course, if $h \le 0$ then $f^{(d)}_{p, 2}=0$.  Proposition~\ref{degbd}
implies that $f^{(d)}_{p, 2}=0$ for $h>p$.  Lascoux gives a much better bound:  $f^{(d)}_{p, 2}=0$ for $h>\sqrt{p}$.
Assume now $1 \le h \le \sqrt{p}$.  Let $S$ be the set of pairs of partitions $(\alpha, \beta)$ such that $\alpha$
has at most $h$ columns, $\beta$ has at most $h$ rows and $\vert \alpha \vert+\vert \beta \vert=p-h^2$.  Associate to
$(\alpha, \beta)$ a new pair of partitions $(\mu, \nu)$ as follows.  Start with a rectangle with $h$ columns and $h+1$
rows.  To get $\mu$, append $\alpha$ to the bottom and $\beta$ to the right.  To get $\nu$, append the dual of $\beta$
to the bottom and the dual of $\alpha$ to the right.  Lascoux's result is then
\begin{displaymath}
f^{(d)}_{p, 2}=\tfrac{1}{2} \sum_{(\alpha, \beta) \in S} s_{\mu} s_{\nu}.
\end{displaymath}
For example, say $p=1$ and $d=2$.  Then $h=1$.  Since $p-h^2=0$ the set $S$ consists of the single pair $(\alpha,
\beta)$ where $\alpha=\beta$ is the zero partition.  The partitions $\mu$ and $\nu$ are both $(1, 1)$ and so we
find $f^{(2)}_{1, 2}=\tfrac{1}{2} s_{(1,1)}^2$, i.e., the quadratic piece of the ideal of the embedding of
$\P(V_1) \times \P(V_2)$ is $\lw{2}{V_1} \otimes \lw{2}{V_2}$.

\subsection{The polynomial $g_p$}

Let $G_p=\Psi(F_p)$ be the cokernel of the map $\Delta F_p \to F_p$.  Since $F_p$ is a finitely generated
$\Delta$-module, $G_p$ is a finite length object of the category $\Sym(\mc{S})$.  The object $G_p$ records precisely
those syzygies that cannot be built out of syzygies on a product of fewer projective spaces.  We let $g_p$ be the
\emph{polynomial} $[G_p]$ (and $g_p^*=[G_p]^*$).  The objects $L^i \Psi F_p$ for $i \ge 1$ are important as well
--- indeed, $[F_p]$ can be recovered from them --- though they are bit less accessible.

We remark that our two main theorems can be rephrased using $f_p$ and $g_p$ so as to look more similar:
Theorem~\ref{thm2} is exactly the statement that $g_p^*$ is a polynomial, while Theorem~\ref{thm3} is exactly the
statement that $f_p^*$ is a rational function.

\subsection{An Euler characteristic}
\label{ss:euler}

Define
\begin{displaymath}
\chi=\sum_{p \ge 0} (-1)^p f_p.
\end{displaymath}
There are only finitely many terms of a given degree and order in the sum, and so it makes sense.  We remark that
$f_0=1$ --- the first term in the resolution of $R$ is always $P=P \otimes \C$ and so $F_0=\C$ for any $(V, L)$.
The main result of this section is an explicit computation of $\chi$.  The notation used in the following proposition
is defined below it.

\begin{proposition}
\label{euler}
We have
\begin{displaymath}
\chi=\left[ \sum_{k=0}^{\infty} \exp(s_{(k)}) \right] \boxtimes \left[ \sum_{\lambda} \kappa_{\lambda}
\exp(s_{\lambda}') \right],
\end{displaymath}
where $\kappa_{\lambda}$ is the rational number $(-1)^{\vert \lambda \vert} \sgn(c_{\lambda}) z_{\lambda}^{-1}$.  The
second sum is taken over all partitions $\lambda$ --- including $\lambda=0$, where the term is 1.
\end{proposition}

Extracting the degree $k$ piece of the above formula yields:

\begin{corollary}
\label{eulercor}
For $k>0$ we have
\begin{displaymath}
\chi^{(k)}=\sum_{p=0}^k \left[ \frac{(-1)^p}{p!} \sum_{\lambda \vdash p} (\# c_{\lambda}) \sgn(c_{\lambda})
\exp(s_{(k-p)} \boxtimes s'_{\lambda}) \right].
\end{displaymath}
The $p=0$ term of the above sum is $\exp(s_{(k)})$.  We have $\chi^{(0)}=1$.
\end{corollary}

We now define notation that will be in place for the rest of the section (and is used in the above proposition).
Let $\lambda$ be a partition of $p$.  We let $c_{\lambda}$ denote the conjugacy class of $S_p$ corresponding to
$\lambda$, normalized so that $\lambda=(1, \ldots, 1)$ corresponds to the identity element, and we let $z_{\lambda}=p!/
\# c_{\lambda}$ be the order of the centralizer of any element of $c_{\lambda}$.  We let $\chi_{\lambda}$
denote the character of $S_p$ corresponding to $\lambda$, normalized so that $\lambda=(1, \ldots, 1)$ corresponds
to the sign character $\sgn$.  The notation $s_{\lambda}$ means what is has meant previously, namely the element
$[\bS_{\lambda}]$ of $K(\mc{S})$; in particular, $s_{(k)}=[\Sym^k]$.  We define $s'_{\lambda}$ to be the element of
$K(\mc{S})$ of degree $p$ which corresponds to the class function on $S_p$ supported on $c_{\lambda}$ and
taking value $z_{\lambda}$ there.  Explicitly,
\begin{displaymath}
s'_{\lambda}=\sum_{\mu \vdash p} \chi_{\mu}(c_{\lambda}) s_{\mu}.
\end{displaymath}
The symbol $\boxtimes$ is the point-wise tensor product:  $s_{\lambda} \boxtimes s_{\mu}$ is computed using the
Littlewood--Richardson rule.
As usual, $W=W_1$ is the object of $\Sym(\mc{S})$ which assigns to $(V, L)$ the tensor product of the $V$'s.
Throughout this section $\lw{i}{W}$ and $\bS_{\lambda}(W)$ refer to point-wise operations in $\Sym(\mc{S})$.
For instance, $\lw{i}{W}$ is the object of $\Sym(\mc{S})$ given by
\begin{displaymath}
(\lw{i}{W})(V, L)=\lw{i}(W(V, L))=\bw{i}{} \left( \bigotimes_{x \in L} V_x \right).
\end{displaymath}
We now begin proving Proposition~\ref{euler}.  We begin with the following.

\begin{lemma}
\label{euler-1}
We have
\begin{displaymath}
\chi = [R] \boxtimes \left( \sum_{p=0}^{\infty} (-1)^p [ \lw{p}{W} ] \right).
\end{displaymath}
\end{lemma}

\begin{proof}
As discussed in the proof of Theorem~\ref{thm:delta-small}, $F_p$ is the homology of the complex $R \boxtimes \lw{i}{W}$
at $i=p-1$.  The formula follows from standard facts about Euler characteristics.
\end{proof}

\begin{lemma}
\label{euler-2}
We have $[R^{(d)}]=\exp(s_{(d)})$, and so $[R]=\sum_{d \ge 0} \exp(s_{(d)})$.
\end{lemma}

\begin{proof}
We have
\begin{displaymath}
R^{(d)}_n(V_1, \ldots, V_n)=\Sym^d(V_1) \otimes \cdots \otimes \Sym^d(V_n).
\end{displaymath}
Thus $[R^{(d)}_n]$ is equal to $\tfrac{1}{n!} [\Sym^d]^n$ in $\Sym^n(K(\mc{S}))$.  The result follows.
\end{proof}

\begin{lemma}
\label{euler-3}
Let $\lambda$ be a partition of $p>0$.  Let $s_{\lambda, n}$ be the class in $\Sym^n(K(\mc{S}))$ of the
functor $(V_1, \ldots, V_n) \mapsto \bS_{\lambda}(V_1 \otimes \cdots \otimes V_n)$.  Then
\begin{displaymath}
\sum_{n=0}^{\infty} s_{\lambda, n}=\frac{1}{p!} \sum_{\mu \vdash p} (\# c_{\mu}) \chi_{\lambda}(c_{\mu})
\exp(s'_{\mu}).
\end{displaymath}
Note that the left side above is nothing other than $[\bS_{\lambda}(W)]$.
\end{lemma}

\begin{proof}
A simple manipulation shows that for any vector spaces $U$ and $V$ we have
\begin{displaymath}
\bS_{\lambda}(U \otimes V)=\bigoplus C_{\lambda \mu \nu} \bS_{\mu}(U) \otimes \bS_{\nu}(V),
\end{displaymath}
where the sum is over all partitions $\mu$ and $\nu$ of $p$, and
\begin{displaymath}
C_{\lambda \mu \nu}= \dim (\Sp_{\lambda} \otimes \Sp_{\mu} \otimes \Sp_{\nu})^{S_p} = \dim \Hom_{S_p}
(\Sp_{\lambda}, \Sp_{\mu} \otimes \Sp_{\nu}).
\end{displaymath}
(This appears as Exercise~6.11(b) in \cite{FH}.)  We therefore have
\begin{displaymath}
\bS_{\lambda}(V_1 \otimes \cdots \otimes V_n)=\bigoplus_{\mu, \nu} C_{\lambda \mu \nu} \bS_{\mu}(V_1) \otimes
\bS_{\nu}(V_2 \otimes \cdots \otimes V_n).
\end{displaymath}
We thus have a recurrence
\begin{displaymath}
s_{\lambda, n}=\frac{1}{n} \sum_{\mu, \nu} C_{\lambda \mu \nu} s_{\mu} s_{\nu, n-1}.
\end{displaymath}
It will now be convenient to switch from working in the degree $p$ piece of $K(\mc{S})$ to working in $K(S_p)$.
The two are in isomorphism via $s_{\lambda}=[\bS_{\lambda}] \leftrightarrow [\Sp_{\lambda}]$.  Thus $s_{\lambda, n}$ can
be regarded as an element of $\Sym^n(K(S_p))$.  Note that the sum $\sum_{\mu} C_{\lambda \mu \nu} s_{\mu}$ is equal
to $[\Sp_{\lambda} \otimes \Sp_{\nu}]$.  We can thus rephrase our last expression as follows.  Let $v_n$ be the column
vector $(s_{\lambda, n})_{\lambda}$ and let $A$ be the matrix $([\Sp_{\lambda} \otimes \Sp_{\mu}] )_{\lambda, \mu}$.
Then
\begin{displaymath}
v_n=\tfrac{1}{n} A v_{n-1}.
\end{displaymath}
We thus have
\begin{displaymath}
\sum v_n=\exp(A) v_0.
\end{displaymath}
Note that $v_0$ has a 1 in the entry $\lambda=(p)$ and a 0 in all other entries.  Indeed, an empty tensor product is
equal to $\C$, so $s_{\lambda, 0}$ is the class of $S_{\lambda}(\C)$ in $\Sym^0(K(S_p))=\Q$; in other words,
$a_{\lambda, 0}=\dim{\bS_{\lambda}(\C)}$.  This is 1 if $\lambda=(p)$ and 0 otherwise.  We therefore find that the
initial vector $v_0$ in the above recurrence is quite simple.  The problem is to determine the exponential of the
matrix $A$.  We will achieve this by diagonalizing $A$.

Let $B$ be the matrix $(\chi_{\lambda}(c_{\mu}))_{\lambda,
\mu}$.  We index by rows first, then columns.  Thus the rows of $B$ are indexed by irreducible characters and the
columns by conjugacy classes; $B$ is the character table of $S_p$.  Let $D$ be the diagonal matrix
given by $D_{\lambda \lambda}=z_{\lambda} \delta_{\lambda}$ where $z_{\lambda}=p!/\# c_{\lambda}$ is the cardinality
of the centralizer of $c_{\lambda}$ and $\delta_{\lambda}$ is the class function on $S_p$ which assigns
$c_{\lambda}$ the value 1 and all other conjugacy classes 0.  We then have the following fundamental
identity
\begin{equation}
\label{eq2}
AB=BD.
\end{equation}
We now explain this identity.  First, we regard the entries of $A$ as class functions, so $A_{\lambda \mu}$ is
the character of $\Sp_{\lambda} \otimes \Sp_{\mu}$.  The entry of the product $AB$ at $(\lambda, \mu)$ evaluated at
$c_{\zeta}$ is thus given by
\begin{displaymath}
\left( \sum_{\nu} A_{\lambda \nu} B_{\nu \mu} \right)(c_{\zeta})
=\sum_{\nu} \chi_{\lambda}(c_{\zeta}) \chi_{\nu}(c_{\zeta}) \chi_{\nu}(c_{\mu})
=\chi_{\lambda}(c_{\zeta}) \sum_{\nu} \chi_{\nu}(c_{\zeta}) \chi_{\nu}(c_{\mu}).
\end{displaymath}
Now, for $g$ and $h$ in $S_p$ the sum $\sum \chi_{\nu}(g) \chi_{\nu}(h)$ is the trace of $(g, h)$ acting on
the representation $\C[S_p]$ of $S_p \times S_p$.  Since this is a permutation representation, the
trace is given by the number of fixed points.  An element $x \in S_p$ is a fixed point if $gxh^{-1}=x$, or
equivalently, if $g=xhx^{-1}$.  Thus the number of fixed points is 0 if $g$ and $h$ are not conjugate, and is
otherwise the size of the centralizer of $g$.  We therefore find
\begin{displaymath}
\chi_{\lambda}(c_{\zeta}) \sum_{\nu} \chi_{\nu}(c_{\zeta}) \chi_{\nu}(c_{\mu})
=\chi_{\lambda}(c_{\mu}) z_{\mu} \delta_{\mu}(c_{\zeta})
=B_{\lambda \mu} D_{\mu \mu}(c_{\zeta}).
\end{displaymath}
This proves \eqref{eq2}.

The equation \eqref{eq2} diagonalizes $A$.  However, for it to be useful we need to compute $B^{-1}$.  This is
straightforward.  Let $C$ be the diagonal matrix given by $C_{\lambda \lambda}=z_{\lambda}^{-1}$.  Then the
orthonormality of characters is precisely the identity
\begin{displaymath}
BCB^t=1
\end{displaymath}
and so $B^{-1}=CB^t$.  (This again uses the fact that all representations of symmetric groups are self-dual,
which is equivalent to their characters being real valued.)

We now find
\begin{displaymath}
\exp(A) v_0=B \exp(B^{-1}AB) B^{-1} v_0=B \exp(D) CB^t v_0.
\end{displaymath}
Simple matrix multiplication now gives the stated formula.
\end{proof}

\begin{lemma}
\label{euler-4}
Let $x$ and $y$ belong to $K(\mc{S})$.  Then $\exp(x) \boxtimes \exp(y)=\exp(x \boxtimes y)$.
\end{lemma}

\begin{proof}
For an object $F$ of $\mc{S}$ let $F'$ be the object of $\Sym(\mc{S})$ given by $(V, L) \mapsto \bigotimes_{x \in L}
F(V_x)$.  Then $\exp([F])=[F']$.  Now, let $x$ and $y$ in $K(\mc{S})$ be given.  Since $\boxtimes$ is additive and
$\exp$ is multiplicative, it suffices to treat the case where $x=[F]$ and $y=[G]$, with $F$ and $G$ in $\mc{S}$.
We then have
\begin{displaymath}
\exp(x) \boxtimes \exp(y)=[F'] \boxtimes [G']=[F' \boxtimes G']=[(F \boxtimes G)']=\exp([F \boxtimes G])
=\exp(x \boxtimes y).
\end{displaymath}
The key fact is the obvious formula $F' \boxtimes G'=(F \boxtimes G)'$.
\end{proof}

The proposition and corollary follow easily from the above lemmas.

\subsection{Examples for small $p$}
\label{ss:smallp}

The main result of \cite{Rubei} states that Segre embeddings satisfy the Green--Lazarsfeld property $N_3$ but not $N_4$.
This means that $F_1$, $F_2$ and $F_3$ are supported exactly in degrees 2, 3 and 4 respectively (and that $F_4$ has
support outside degree 5).  From this, we deduce the following equalities:
\begin{displaymath}
f_1=-\chi^{(2)}, \qquad f_2=\chi^{(3)}, \qquad f_3=-\chi^{(4)}, \qquad f_4^{(5)}=\chi^{(5)}.
\end{displaymath}
These values have been computed in Proposition~\ref{euler}.  They are listed explicitly, and in simplified form,
in Figure~\ref{fig1} (other than $f_4^{(5)}$).  We explain how the value for $f_1$ given in the figure was derived, the
values of $f_2$ and $f_3$ being gotten in a similar fashion.  Proposition~\ref{euler} gives
\begin{displaymath}
f_1=-\exp(s_{(2)})+\exp(s_{(1)} \boxtimes s'_{(1)}) - \tfrac{1}{2} \left( \exp(s'_{(1,1)})-\exp(s'_{(2)}) \right).
\end{displaymath}
We have $s'_{(1)}=s_{(1)}$, while $s'_{(1,1)}=s_{(2)}+s_{(1,1)}$ and $s'_{(2)}=s_{(2)}-s_{(1,1)}$.  Now, the
product $s_{(1)} \boxtimes s_{(1)}$ is just the usual product in $K(\mc{S})$, i.e., it is the class of the functor
$V \mapsto \Sym^1(V) \otimes \Sym^1(V)$.  This, of course, is equal to $s_{(2)}+s_{(1,1)}$.  We thus find
\begin{displaymath}
f_1=\tfrac{1}{2} \exp(s_{(2)}+s_{(1,1)}) + \tfrac{1}{2} \exp(s_{(2)}-s_{(1,1)}) - \exp(s_{(2)}).
\end{displaymath}
This is the value given in the figure.

We have previously stated (without proof) that the equation for $\P^1 \times \P^1$ in $\P^3$ generates $F_1$ as a
$\Delta$-functor.  This implies that $g_1$ is the order two piece of $f_1$.  Similarly, we have stated
(without proof) that any non-zero syzygy for $\P^1 \times \P^2$ in $\P^5$ generates $F_2$.  This implies that $g_2$ is
the order two piece of $f_2$.  Finally, the order two piece of $g_3$ is the same as that of $f_3$;
we are unaware if $g_3$ has any terms of higher order.  These remarks explain the values of $g_1$, $g_2$
and $g_3$ given in the figure.

\begin{figure}
\begin{displaymath}
\begin{array}{|l|}
\hline \\[-11pt]
s=[\Sym^2], \quad w=[\lw{2}{}] \\[2pt]
\hline \\[-10pt]
f_1=\tfrac{1}{2} e^{s+w} + \tfrac{1}{2} e^{s-w} - e^s \\[2pt]
\hline \\[-10pt]
g_1=\tfrac{1}{2} w^2 \\[2pt]
\hline\hline \\[-11pt]
s=[\Sym^3], \quad w=[\lw{3}{}], \quad t=[\bS_{(2, 1)}] \\[2pt]
\hline \\[-10pt]
f_2=\tfrac{1}{3} e^{s+w+2t} - \tfrac{1}{3} e^{s+w-t} - e^{s+t} + e^s \\[2pt]
\hline \\[-10pt]
g_2=wt \\[2pt]
\hline\hline \\[-11pt]
s=[\Sym^4], \quad w=[\bw{4}{}], \quad a=[\bS_{(3,1)}], \quad b=[\bS_{(2,2)}],
\quad c=[\bS_{(2,1,1)}] \\[2pt]
\hline \\[-10pt]
f_3 =
\tfrac{1}{8} e^{s+w+3a+2b+3c}-\tfrac{1}{8} e^{s+w-a+2b-c}+\tfrac{1}{4} e^{s-w-a+c} -\tfrac{1}{4} e^{s-w+a-c} \\[3pt]
\hskip 5ex + \tfrac{1}{2} e^{s+b-c} - \tfrac{1}{2} e^{s+2a+b+c} + e^{s+a} - e^s. \\[2pt]
\hline \\[-10pt]
g_3=aw+\tfrac{1}{2}c^2 \quad (+?) \\[2pt]
\hline
\end{array}
\end{displaymath}
\caption{Values of $f_p$ and $g_p$ for $p$ small.}
\label{fig1}
\end{figure}

\section{Questions and problems}
\label{s:ques}

\newcounter{quescount}
\def\ques{\par\vskip 1ex\noindent ({\bf \refstepcounter{quescount}\thequescount})\ }

\ques
Are finitely generated twisted commutative algebras noetherian?  We proved this for algebras generated in order 1, and
can also prove it for certain algebras in order 2.

\ques
Are finite level $\Delta$-ideals finitely generated?  More generally, are the $\Delta$-modules $F_p$ for a finite level
$\Delta$-scheme concentrated in finitely many degrees?

\ques
Let $M$ be a finitely generated module over a twisted commutative algebra finitely generated in order 1.  We have
shown that $H_M(t)$ is a polynomial of $t$ and $e^t$.  Our proof shows that the maximal power of $e^t$ is related to
the number of rows in $M$.  How else does the form of $H_M(t)$ relate to the structure of $M$?

\ques
Our series $f_p$ forgets the $S_n$-equivariance on $F_{p,n}$.  Can one modify $f_p$ to retain this information, and
still have something resembling a rationality result?

\ques
Is the series $f_p$ a polynomial in the $s_{\lambda}$ and the $e^{\pm s_{\lambda}}$?  This
does not follow from what we have proved, but one might hope that it is true based on some of our results.  In fact,
based on the computations of $f_1$, $f_2$ and $f_3$ for the Segre, one might hope for a stronger statement:  $f_p$ is
a polynomial in only the $e^{\pm s_{\lambda}}$.  We suspect this is false, but do not know.

\ques
Compute $f_p$ and $g_p$ for more values of $p$ or for $\Delta$-schemes other than the Segre.  We have given an
algorithm to do this, but it is too inefficient to use.  It would be particularly interesting to compute $f_4$ for the
Segre since this is the first place where the Green--Lazarsfeld property fails and the value is not given by the Euler
characteristic formula.

\ques
Is the series $\sum_{i \ge 0} (-1)^i [L^i \Psi F_p]^* q^i$ a rational function of $q$ and the $s_{\lambda}$?  This
series contains more information than $f_p$ and $g_p$, since one can recover $f_p^*$ by applying $\Phi$ and setting
$q=1$, and one can recover $g_p^*$ by setting $q=0$.


\begin{thebibliography}{[ERT]}

\bibitem[B]{Barratt}
M.G.~Barratt, \emph{Twisted Lie algebras}, in Geometric applications of homotopy theory (Proc. Conf., Evanston, Ill.,
1977), II, 9--15, Lecture Notes in Math., 658, Springer, Berlin, 1978.

\bibitem[D]{Deligne}
P.~Deligne, \emph{Cat\'egories tannakiennes}, The Grothendieck Festschrift, Vol.~II, 111--195,
Progr.\ Math., 87, Birkh\"auser Boston, Boston, MA, 1990. 

\bibitem[DK]{DraismaKuttler}
J.~Draisma and J.~Kuttler, \emph{Bounded-rank tensors are defined in bounded degree}, arXiv:1103.5336, 2011.

\bibitem[E]{Eisenbud}
D.~Eisenbud, \emph{Commutative algebra},
Graduate Texts in Mathematics, {\bf 150}, Springer--Verlag, New York, 1995.

\bibitem[ERT]{ERT}
D.~Eisenbud, A.~Reeves and B.~Totaro, \emph{Initial ideals, Veronese subrings, and rates of algebras},
Adv. Math. {\bf 109} (1994), no.~2, 168--187. 

\bibitem[FH]{FH}
W.\ Fulton and J.\ Harris, \emph{Representation Theory: A First Course}, Graduate Texts in Mathematics, {\bf 129},
Springer--Verlag, New York, 1991.

\bibitem[GS]{GinzburgSchedler}
V.~Ginzburg and T.~Schedler, \emph{Differential operators and BV structures in noncommutative geometry},
Selecta Math. (N.S.) {\bf 16} (2010), no.~4, 673--730. 

\bibitem[J]{Joyal}
A.~Joyal, \emph{Foncteurs analytiques et esp\'eces de structures}, in Combinatoire \'enum\'erative (Montreal, Que.,
1985), 126--159, Lecture Notes in Math., 1234, Springer, Berlin, 1986.

\bibitem[L]{Lascoux}
A.~Lascoux, \emph{Syzygies des vari\'et\'es d\'eterminantales}, Adv.\ in Math.\ {\bf 30} (1978), 202--237.

\bibitem[LW]{LW}
J.M.~Landsberg and J.~Weyman, \emph{On tangential varieties of rational homogeneous varieties},
J.~Lond.\ Math.\ Soc.\ {\bf 76} (2007) (2), 513--530.

\bibitem[LW2]{LW2}
J.M.~Landsberg and J.~Weyman, \emph{On the ideals and singularities of secant varieties of Segre varieties},
Bull.\ London Math.\ Soc.\ {\bf 39} (2007) no.~4, 685--697.

\bibitem[PW]{PW}
P.~Pragacz and J.~Weyman, \emph{Complexes associated with trace and evaluation.  Another approach to Lascoux's
resolution}, Adv.\ in Math.\ {\bf 57} (1985), 163--207.

\bibitem[Ra]{Raicu}
C.~Raicu, \emph{The GSS conjecture}, arXiv:1011.5867, 2010.

\bibitem[Ru]{Rubei}
E.~Rubei, \emph{Resolutions of Segre embeddings of projective spaces of any dimension},
J.~of Pure and Applied Algebra {\bf 208} (2007), 29--37.

\end{thebibliography}
\end{document}